\documentclass[twoside]{amsart}
\usepackage{amssymb}
\usepackage{amsmath}
\usepackage{amsthm}
\usepackage{graphicx}
\usepackage{color}
\usepackage[all]{xy}

\newtheorem{theorem}{Theorem}[section]
\newtheorem{proposition}[theorem]{Proposition}
\newtheorem{corollary}[theorem]{Corollary}

\newtheorem{lemma}[theorem]{Lemma}

\theoremstyle{definition}
\newtheorem{definition}[theorem]{Definition}
\newtheorem{example}[theorem]{Example}
\theoremstyle{remark}

\newtheorem{conjecture}[theorem]{Conjecture}

   \newenvironment{eq}{\begin{equation}}{\end{equation}}

\newcommand{\cO}{{\mathcal O}}
\newcommand{\cQ}{{\mathcal Q}}

\newcommand{\cV}{{\mathcal V}}

\newcommand{\cU}{{\mathcal U}} 

\newcommand{\g}{{\mathfrak g}}

\newcommand{\fb}{{\mathfrak b}}
\newcommand{\fu}{{\mathfrak u}}

\newcommand{\Q}{\mathbb Q}
\newcommand{\Z}{\mathbb Z}
\newcommand{\C}{\mathbb C}
\newcommand{\R}{\mathbb R}
\newcommand{\proj}{\mathbb P}
\newcommand{\pu}{\mathbb P ^1}

\newcommand{\bQ}{\mathbb Q}
\newcommand{\bZ}{\mathbb Z}
\newcommand{\bR}{\mathbb R}
\newcommand{\bP}{\mathbb P}

\newcommand{\bA}{\mathbb A}
\newcommand{\bB}{\mathbb B}
\newcommand{\bC}{\mathbb C}
\newcommand{\bD}{\mathbb D}
\newcommand{\bE}{\mathbb E}
\newcommand{\bF}{\mathbb F}
\newcommand{\bG}{\mathbb G}

\newcommand{\ra}{\rightarrow}
\newcommand{\lra}{\longrightarrow}
\newcommand{\iso}{\simeq}

\DeclareMathOperator{\Spec}{Spec}
\DeclareMathOperator{\Symm}{Symm}
\DeclareMathOperator{\Proj}{Proj}

\DeclareMathOperator{\Nef}{Nef}
\DeclareMathOperator{\Ess}{Ess}

\DeclareMathOperator{\Mov}{Mov}

\DeclareMathOperator{\Pic}{Pic}

\DeclareMathOperator{\Cl}{Cl}

\DeclareMathOperator{\codim}{codim}

\newcommand{\bubla}{\blacksquare}
\newcommand{\bured}{{\color{red}\blacktriangle}\normalcolor}
\newcommand{\bublu}{{\color{blue}\bigstar}}
\newcommand{\bugre}{{\color{green}\blacklozenge}}

\begin{document}


\title[Symplectic contractions]{4-dimensional symplectic contractions} 
\author[Andreatta \& Wi\'sniewski]{Marco Andreatta and Jaros\l{}aw
  A. Wi\'sniewski}

\thanks{}

\address{Dipartimento di Matematica, Universita di Trento, I-38050
Povo (TN)} \email{marco.andreatta@unitn.it}

\address{Instytut Matematyki UW, Banacha 2, PL-02097
Warszawa} \email{J.Wisniewski@mimuw.edu.pl}

\thanks{ The first author was supported by the Italian PRIN; he thanks
  also the Institute of Mathematics of Warsaw University for
  hospitality. The second author was supported by Polish MNiSzW grants
  N N201 2653 33 and N N201 4206 39 as well as Fondazione Bruno
  Kessler at CIRM in Trento. Thanks to Andrzej Weber for stimulating
  discussions.} 

\begin{abstract}
  Local symplectic contractions are resolutions of singularities which
  admit symplectic forms.  Four dimensional symplectic contractions
  are (relative) Mori Dream Spaces. In particular, any two such
  resolutions of a given singularity are connected by a sequence of
  Mukai flops. We discuss the cone of movable divisors on such a
  resolution; its faces are determined by curves whose loci are
  divisors, we call them essential curves. The movable cone is divided
  into nef chambers which are related to different resolutions; this
  subdivision is determined by classes of 
  1-cycles.  We also study schemes parametrizing minimal essential
  curves and show that they are resolutions, possibly non-minimal, of
  surface Du Val singularities.  Some examples, with an exhaustive
  description, are provided.
\end{abstract}

\maketitle

\section{Introduction}
In the paper we consider {\it local symplectic contractions} of
$4$-folds. That is, we deal with maps $\pi : X \ra Y$ where
\begin{itemize}
\item $X$ is a smooth complex $4$-fold with a closed holomorphic
  $2$-form, non-degenerate at every point,
\item $Y$ is an affine (or Stein) normal variety,
\item $\pi$ is a birational projective morphism.
\end{itemize}

In dimension 2 symplectic contractions are classical and they are
minimal resolutions of Du Val singularities. In fact, any symplectic
contraction can be viewed as a special symplectic resolution of a
symplectic normal singularity. 

General properties of symplectic contractions (in arbitrary dimension)
have been considered in a number of papers published in the last
decade: \cite{BeauvilleSympSing}, \cite{Verbitsky}, \cite{Namikawa},
\cite{KaledinMcKay}, \cite{KaledinResQSing}, \cite{Wierzba},
\cite{FuNamikawa}, \cite{GinzburgKaledin}, \cite{FuMukaiFlops},
\cite{HassettTschinkel}, \cite{Bellamy}, \cite{LehnSorger}, to mention
just a few; see also \cite{FuSurvey} for more references and a review
on earlier developments in this subject. Let us just recall two
beautiful results about symplectic contractions: these maps are
semismall, \cite{Wierzba}, and the McKay correspondence holds for those
symplectic contractions which are resolutions of quotient symplectic
singularities, \cite{KaledinMcKay}, \cite{GinzburgKaledin}.  However,
in dimension 4 and higher, apart from the description in codimension 2,
\cite{Wierzba}, not much is known about the fine geometrical structure
of these morphisms which is the problem we want to tackle in the
present paper.

The 4-dimensional small case (i.e.~when $\pi$ does not contract a
divisor) is known by \cite[Thm.~1.1]{WierzbaWisniewski}.  Using this
result we first prove in Section \ref{MDSstructure} that $X$ is a Mori
Dream Space over $Y$, as defined in \cite{HuKeel}.  In short, every
movable divisor of $X$ (over $Y$) can be made nef and semiample after
a finite number of small modifications (flops); see also
\cite{Wierzba-unpublished} and \cite[Thm.~1.2]{WierzbaWisniewski},
where a version of this result was announced. In Theorem
\ref{cone-structure} we describe the cone of movable divisors of
$X/Y$; by Theorem \ref{dualityMovEss} it can be also described in
terms of classes of curves contracted in the divisorial locus. As a
result, these objects are (with one possible exception) as in the case
of the contraction to the nilpotent cone (we recall this situation in
the Appendix). Next, we give properties of the subdivision of the cone
of movable divisors into chambers corresponding to nef cones of small
modifications of $X$, see Theorem \ref{hyperplane}.

In Section \ref{rational-curves}, following the approach introduced in
\cite{Wierzba} and subsequently in \cite{SolaCondeWisniewski}, we
study families of rational curves (i.e. irreducible components of the
Chow scheme) in $X/Y$. In Theorem \ref{DuValsing} we prove that they
are resolutions of Du Val singularities, possibly non-minimal; the
discrepancies of this resolution depend on the rank of the evaluation
map, see Proposition \ref{discrepancy}, and on the modification of
$X$, see Lemma \ref{changeV}. We also show that studying 4-dimensional
symplectic resolutions implies understanding arbitrary dimensional
case in codimension 4, via the argument of general intersection of a
suitable number of divisors, we call it vertical slicing, see
Corollary\ref{vertical-slicing}.

In Section \ref{examples} we study known examples of resolutions of
quotient symplectic singularities and we describe explicitly their
movable cones and their families of rational curves. In particular, we
describe explicitly the division of the movable cone of a symplectic
resolution of $\bC^4/(\bZ_{n+1}\wr\bZ_2)$ into nef cones which are
associated with different resolutions of this singularity. The main
result of this section, Theorem \ref{cone-division}, gives a
description of the cone of movable divisors into nef chambers by
hyperplanes defined by classes of contracted curves contained in the
divisorial locus. The following Theorem summarizes our results in this
particular case and, in fact, serves as a model of the situation that
we expect to get while dealing with arbitrary local symplectic
contraction in dimension 4 (for definitions and other examples see
Section 4).

\begin{theorem}\label{example-theorem}
  Let $\bA_1\oplus\bA_n$ be a decomposable root system in
  $V\iso\bR^{n+1}$ with simple (positive) roots denoted by $e_0$ and
  $e_1,\dots, e_n$. Let $\Mov=\langle e_0, e_1,\dots, e_n\rangle^\vee$
  be the cone in $V^\vee$ dual to the cone spanned by the roots
  $e_0,e_1,\dots, e_n$. Then the symplectic resolutions of the
  singularity $\bC^4/(\bZ_{n+1}\wr\bZ_2)$ are in bijective relation
  with maximal dimensional cones obtained by cutting the cone $\Mov$
  with hyperplanes $(e_0-e_{ij})^\perp$ where vectors
  $e_{ij}=e_i+\cdots+e_j$, for $1\leq i\leq j\leq n$, are all positive
  roots of the root system $\bA_n$.
\end{theorem}

\section{Notation and preliminaries}

\subsection{Symplectic contractions}\label{defSymContr}
A holomorphic $2$-form $\omega$ on a smooth variety is called {\bf
  symplectic} if it is closed and non-degenerate at every point. A
{\bf symplectic variety} is a normal variety $Y$ whose smooth part
admits a holomorphic symplectic form $\omega_Y$ such that its pull
back to any resolution $\pi: X \ra Y$ extends to a holomorphic
$2$-form $\omega_X$ on $X$. We call $\pi$ a {\bf symplectic
  resolution} if $\omega_X$ is non degenerate on $X$, i.e. it is a
symplectic form.  More generally, a map $\pi:X \ra Y$ is called a {\bf
  symplectic contraction} if $X$ is a symplectic manifold, $Y$ is
normal and $\pi$ is a birational projective morphism. If moreover $Y$
is affine we will call $\pi:X \ra Y$ a {\bf local symplectic
  contraction} or {\bf local symplectic resolution}. The following
facts are well known, see the survey paper \cite{FuSurvey} and
references therein.

\begin{proposition} Let $Y$ be a symplectic variety and $\pi: X \ra Y$
  be a resolution. Then the following statement are equivalent: (i)
  $\pi^*K_Y=K_X$, (ii) $\pi$ is symplectic, (iii) $K_X$ is
  trivial, (iv) for every symplectic form on $Y_{reg}$ its pull-back
  extends to a symplectic form on $X$.
\end{proposition}

\begin{theorem}\label{semis} A symplectic resolution $\pi : X\ra Y$
  is semismall, that is for every closed subvariety $Z \subset X$ we
  have $2 \codim Z \geq \codim\pi(Z) $.  If equality holds $Z$ then is
  called a maximal cycle.
\end{theorem}

\begin{corollary}\label{iso-codim1}
  Any two symplectic resolutions $\pi_i : X_i\ra Y$, where $i=1,\ 2$,
  are isomorphic in codimension 1.
\end{corollary}

\begin{proof}
  By \ref{semis} every exceptional divisor of a symplectic resolution
  $\pi_i$ is mapped to a codimension 2 set in $Y$. On the other hand,
  any symplectic resolution of $Y$ is uniquely determined in
  codimension 2 as the resolution of the surface Du Val singularities.
\end{proof}

\begin{example}\label{HilbertChow}Let $S$ be a smooth surface (proper or
  not). Denote by $S^{(n)}$ the {\it symmetric product} of $S$, that
  is $S^{(n)} = S^n/\sigma_n$, where $\sigma_n$ is the symmetric group
  of peremutations of $n$ elements.  Let also $ Hilb^n(S)$ be the {\it
    Hilbert scheme} of $0$-cycles of degree $n$. A classical result
  (c.f.~\cite{Fogarty}) says that $Hilb^n(S)$ is smooth and that
  $\tau: Hilb^n(S) \ra S^{(n)}$ is a crepant resolution of
  singularities. We will call it a Hilb-Chow map.

  Suppose now that $S\ra S'$ is a resolution of a Du Val singularity
  which is of type $S'=\C^2/H$ with $H<SL(2,\C)$ a finite group. 
  Then the composition $Hilb^n(S) \ra
  S^{(n)}\ra (S')^{(n)}$ is a local symplectic contraction.

  We note that $(S')^{(n)}$ is a quotient singularity with respect to
  the action of the wreath product $H\wr\sigma_n=(H^n)\rtimes\sigma_n$
  (the group $\sigma_n$ permutes factors in $H^n=H^{\times n}$).
\end{example}

\subsection{Mori Dream Spaces}
\label{MDS}
Let us recall basic definitions regarding Mori Dream Spaces. For more
information we refer to \cite{HuKeel} or \cite{ADHL}.  Our definitions
are far from general but sufficient for our particular local set-up.
We assume that
\begin{enumerate}
\item $\pi:X \ra Y$ is a projective morphism of normal varieties with
  connected fibers, that is $\pi_*\cO_X=\cO_Y$ and $Y=\Spec A$ is
  affine,
\item $X$ is locally factorial and $\Pic(X/Y)=\Cl(X/Y)$ is a finitely
  generated abelian group so that $N^1(X/Y)=\Pic(X/Y)\otimes\Q$ is a
  finite dimensional vector space
\end{enumerate}

By $\Nef(X/Y)\subset N^1(X/Y)$ we understand the closure of the convex
cone spanned by the classes of relatively ample bundles, while by
$\Mov(X/Y)\subset N^1(X/Y)$ we understand the closure of the convex
cone spanned by the classes of linear systems which have no fixed
components.  That is, a class of a $\Q$-divisor $D$ is in $\Mov(X/Y)$
if the linear system $|mD|$ has no fixed component for $m\gg 0$.
Clearly $\Nef(X/Y)\subset \Mov(X/Y)$ and we call $\Nef(X/Y)$ the Mori
cone (or chamber) of $X$.  By the relative version of Kleiman's
criterion of ampleness the cone $\Nef(X/Y)$ is strictly convex.

The following is a version of \cite[Def.~1.10]{HuKeel}.

\begin{definition} \label{defMDS} In the above situation we say that
  $X$ is a Mori Dream Space (MDS) over $Y$ if in addition
  \begin{enumerate}
  \item the cone $\Nef(X/Y)$ is spanned by the classes of finitely
    many semi-ample line bundles:
  \item there is a finite collection of small $\Q$-factorial
    modifications (SQM) over $Y$, $f_i:X\ -\ra X_i$ such that $X_i
    \lra Y$ satisfies the above assumptions and $\Mov(X/Y)$ is the
    union of the strict transforms $(f_i)_*(\Nef(X_i/Y))$
  \end{enumerate}
\end{definition}

Here, since $X$ and $X_i$ are both $\Q$-factorial and isomorphic in
codimension 1, we can identify, via $f_i$, $N^1(X/Y)$ with $N^1(X_i/Y)$
and $\Mov(X/Y)$ with $\Mov(X_i/Y)$. By abuse we will just write
$N^1(X/Y)=N^1(X_i/Y)$ and $\Mov(X/Y)=\Mov(X_i/Y)$.

\begin{example}\label{MukaiFlop}
  Take the $\C^*$ action on $\C^r\times\C^r$ with coordinates $(x_i,y_j)$
  and weights $1$ for $x_i$'s and $-1$ for $y_j$'s. Using these
  weights we define a $\Z$-grading of the polynomial ring and write
  $\C[x_i,y_j] = \bigoplus_{m\in\Z}A_m$. The quotient
  $\widehat{Y}=\Spec A_0$ is a toric singularity which, in the
  language of toric geometry, is associated with a cone spanned by $r$
  vectors $e_i$ and $f_j$, in the lattice of rank $2r-1$, with one
  relation $\sum e_i=\sum f_j$. The result is the cone over Segre
  embedding of $\bP^{r-1}\times \bP^{r-1}$. Consider $A_+=
  \bigoplus_{m\geq 0} A_m$ and $A_-=\bigoplus_{m\leq 0}A_m$, and
  define two varieties over $\widehat{Y}$:
  $$\widehat{X}_\pm=\Proj_{A_0}A_\pm\lra \widehat{Y}$$
  Both, $\widehat{X}_+$ and $\widehat{X}_-$, are smooth because, as
  toric varieties, they are associated with two unimodular
  triangulations of the cone in question: one in which we omit
  consecutive $e_i$'s, the other in which we omit $f_j$'s.  The affine
  pieces of covering are of type $\Spec\C[x_i/x_k,x_ky_j]$, where
  $k=1,\dots, r$, for $\widehat{X}_+$ and similar for $\widehat{X}_-$.
  The two resolution $\widehat{X}_+ \rightarrow \widehat{Y}
  \leftarrow\widehat{X}_-$ form two sides of the so-called Atiyah
  flop.

  \par\medskip

  Consider an ideal $I=\left(\sum_i x_iy_i\right) \triangleleft
  \C[x_i,y_j]$ generated by a $\C^*$ invariant function (degree 0) and
  its respective counterparts $I_0\cap A_0\triangleleft A_0$,
  $I_+\triangleleft A_+$ and $I_-\triangleleft A_-$.  We set $Y=\Spec
  A_0/I_0$, $X_+=\Proj_{A_0}A_+/I_+$ and $X_-=\Proj_{A_0}A_-/I_-$ and
  call the resulting diagram ${X}_+\rightarrow {Y} \leftarrow{X}_-$
  {\bf Mukai flop}. By abuse, by the same name we will call a
  respective diagram which is locally isomorphic to the present one in
  the analytic or formal category. The variety $Y$ is symplectic since
  the form $\omega=\sum_i (dx_i \wedge dy_i)$ on $\C^r\times\C^r$
  descends to a symplectic form on $Y$. The varieties $X_\pm$ are its
  small symplectic resolutions and $\C[x_i,y_j]/I$ is their Cox ring,
  \cite{ADHL}.  We note that $\Spec(\C[x_i,y_j]/I)$ is the cone over
  the incidence variety of points and hyperplanes in $\bP^{r-1} \times
  (\bP^{r-1})^*$. Finally, we note that the movable cone $\Mov(X_\pm)$
  is the whole line $N^1(X_\pm)$, hence it is not strictly convex.


\end{example}

\section{Local symplectic contractions in dimension $4$.}

\subsection{MDS structure} \label{MDSstructure} In this section $\pi:
X \ra Y$ is a local symplectic contraction, as defined in
\ref{defSymContr}, and $\dim X = 4$.  By the semismall property (see
Theorem \ref{semis}), the fibers of $\pi$ have dimension less or equal
to $2$.  We will denote with $0$ the unique (up to shrinking $Y$ to a
smaller affine set) point such that $\dim \pi^{-1}(0) =2$. If $\pi$ is
divisorial then the general non trivial fiber has dimension $1$.  

We note that $\pi: X \ra Y$ satisfies the assumptions stated in
\ref{MDS}. In particular, because $R^i\pi_*\cO_X=0$ for $i>0$, it
follows that $N^1(X/Y)$ is a finite dimensional vector space. By
$N_1(X/Y)$ we denote the $\Q$ vector space of 1-cycles proper over
$Y$, modulo numerical equivalence (c.f.~\cite[Ex.~2.16]{KollarMori}).
Then $N_1(X/Y)$ and $N^1(X/Y)$ are dual via the intersection pairing.

We start by recalling the following theorem from
\cite[Thm.~1.1]{WierzbaWisniewski}.

\begin{theorem}\label{WW} Suppose that $\pi$ is small (i.e.~it does
  not contract any divisor). Then $\pi$ is locally analytically
  isomorphic to the collapsing of the zero section in the cotangent
  bundle of $\proj ^2$. Therefore $X$ admits a Mukai flop as described
  in example \ref{MukaiFlop}
\end{theorem}

The above theorem, together with Matsuki's termination of
4-dimensional flops, see \cite{Matsuki}, is the key ingredient in the
proof of the following result. Similar results are
\cite[Thm.~1.2]{WierzbaWisniewski}, \cite[Thm.~4.1]{FuNamikawa}, and
in \cite{Wierzba-unpublished}, as well as in \cite{BurnsHuLuo}. The
classical references for the Minimal Model Program (MMP), which is the
framework for this argument, are \cite{KawamataMatsudaMatsuki} and
\cite{KollarMori}; the finitness argument is in \cite{KawamataMatsuki}.

\begin{theorem}\label{4dimMDS} Let $\pi: X \ra Y$ be a 4-dimensional local 
  symplectic contraction and let $\pi^{-1}(0)$ be its only
  2-dimensional fiber.  Then $X$ is a Mori Dream Space over $Y$.
  Moreover any SQM model of $X$ over $Y$ is smooth and any two of them
  are connected by a finite sequence of Mukai flops whose centers are
  over $0\in Y$. In particular, there are only finitely many non
  isomorphic (local) symplectic resolutions of $Y$,
  c.f.~\cite[Thm.~4.1]{FuNamikawa}.
\end{theorem}

\begin{proof} We note that, by the (relative) non-vanishing theorem
  (Theorem 3.4 in \cite{KollarMori}), the linear and numerical
  equivalence over $Y$ are the same hence $\Pic(X/Y)$ is a finitely
  generated free abelian group (a lattice).  Note also that, since we
  are in a relative situation over $Y$ via a birational map $\pi$, we
  can assume that a $\pi$-nef line bundle is also $\pi$-big: in fact a
  nef bundle plus a big one is big and in our situation the trivial
  bundle is $\pi$-big.  By the (relative) Kawamata-Shokurov
  basepoint-free theorem (Theorem 3.24 in \cite{KollarMori}) every nef
  divisor on $X$ (we drop from now on the suffix $\pi$) is also
  semiample. On the other hand, the (relative) rationality theorem (we
  use Theorem 3.25(2) in \cite{KollarMori}, we choose an effective
  $\bQ$-divisor $\Delta$ such that $-\Delta$ is $\pi$-ample which is
  possible by Kodaira's lemma) asserts that $\Nef(X/Y)$ is rational
  polyhedral.

  Next we claim that $X$ satisfies the second property of definition
  \ref{defMDS}. To achieve this, first we prove that for any $\pi_i:
  X_i\ra Y$ which is a small $\Q$-factorial modification of $X$, there
  exists a finite number of Mukai flops $X-\ra\cdots-\ra X_i$ which
  modifies $X$ to $X_i$. For this, we take $D_i$ which ample on $X_i$
  and its strict transform is a movable divisor on $X$. We may assume
  that $X$ is not isomorphic (over $Y$) to $X_i$ hence $D_i$ is not
  nef. We look for extremal rays which have negative intersection with
  it.  They have to be associated with small contractions because they
  have to be in the base point locus of the divisor. By theorem
  \ref{WW} these are contractions of a $\proj ^2$ which can be flopped
  (via a Mukai flop) so that the result remains smooth.  The process
  has to finish by the "termination of flops" (relative to the chosen
  movable divisor), which is the main result in \cite{Matsuki}.
  Therefore, after a finite number of flops, over a variety $X_i'$ the
  strict transform of $D_i$ becomes a nef divisor, which is semiample
  (by the basepoint-free Theorem 3.24 in \cite{KollarMori}). The
  induced regular morphism $$X_i'\ra\Proj_Y(\bigoplus_{m\geq
    0}\Gamma(X_i',\cO(mD_i)))=X_i$$ is in fact an isomorphism because
  it induces an isomorphism $N^1(X_i'/Y)=N^1(X_i/Y)$.
  
  It remains to prove that the subdivision of ${\Mov}(X/Y)$ into the
  nef subcones of different SQM's is finite. The argument is the same
  as in the proof of the main theorem of \cite{KawamataMatsuki},
  pp.~596---597, where we refer the reader for details. Namely, we
  choose an effective $\bQ$ divisor $\Delta$ such that $-\Delta$ is
  $\pi$-ample and the pair $(X,\Delta)$ is Kawamata log-terminal, or
  klt. In fact the pair $(X_i,\Delta_i)$ is klt for any small
  $\bQ$-factorial modification of $X$, $\pi_i: X_i\ra Y$ with
  $\Delta_i$ the strict transform of $\Delta$. By Kawamata rationality
  theorem, see \cite[Thm.~4-1-1]{KawamataMatsudaMatsuki} or
  \cite[Thm.~3.5]{KollarMori}, the nef threshold for any ample divisor
  on $X_i$ has a bounded denominator which is impossible if the number
  of Mori chambers in $\Mov(X/Y)$ is infinite.
\end{proof}

\subsection{Flopping classes}\label{flopping-clases}

The following theorem provides a somewhat more refined description of 
the division of the cone $\Mov(X/Y)$.
 
\begin{theorem}\label{hyperplane}
  Let $\pi: X\ra Y$ be a local 4-dimensional symplectic
  contraction. The subdivision of ${\Mov}(X/Y)$ into the nef subcones
  of different SQM models is obtained by cutting $\Mov(X/Y)$ with
  hyperplanes. That is, the union of the interiors of nef cones of all
  SQM models of $X$ is equal to ${\Mov}(X/Y) \setminus \bigcup
  \lambda_i^\perp $, where $\lambda_i$'s are classes in
  $N_1(X/Y)$. Moreover the number of these hyperplanes is finite as
  the number of the chambers is finite as well.
\end{theorem}

The $\lambda_i$'s in the above theorem are determined up to
multiplicity and they will be called the {\em flopping classes} of
$X$. We think about $\lambda_i$'s as vectors in the dual of
$N^1(X/Y)$, which is $N_1(X/Y)$, supporting hyperplanes
$\lambda_i^\perp=\{v\in N^1(X/Y): \lambda\cdot v=0\}$. We assume that
$\lambda_i^\perp$ have non-emty intersection with the interior of the
cone $\Mov(X/Y)$. If $\pi': X'\ra Y$ is a small $\Q$-factorial
modification of $X$ then, via the identification $N^1(X/Y)=N^1(X'/Y)$,
the hyperplanes $\lambda_i^\perp$'s are well defined for $X'$ and we
call them the flopping classes of $X'$ as well. We note, however, that
there is no natural identification of 1-cycles in $X$ and $X'$.

\medskip
\begin{proof} Take a cone $F$ which is a face the nef cone of some
  SQM, say $X$, in the interior of ${\Mov}(X/Y)$. Thus the exceptional
  locus of the contraction of $X$ associated with the face $F$
  consists of a number of disjoint copies of $\bP^2$. This statement
  requires not only the theorem \ref{WW} but also the argument proving
  that $\bP^2$'s are disjoint; the latter is in the proof of (3.2) of
  \cite{WierzbaWisniewski}.

  Let $W=\lambda^\perp\cap\Nef(X/Y)\supset F$ be a facet (maximal
  dimensional face) of $\Nef(X/Y)$ which is also a wall of the
  subdivision of $\Mov(X/Y)$ into nef chambers associated to different
  SQM's. The class $\lambda\in N_1(X/Y)$ can be realized by rational
  curves contracted by an extremal contraction of $X$ which factors
  the contraction associated to $F$.

  If $W'$ is another facet of $\Nef(X/Y)$ containing $F$ then loci of
  curves determining $W$ and $W'$ are disjoint. Thus the flop $X-\ra
  X'$ with respect to the wall $W'$ does not affect curves whose class
  is $\lambda$. That is, $W=\lambda^\perp\subset N^1(X'/Y)$ is
  determined by the class of a curve. This will remain true for any
  further flop associated to any wall containing the cone $F$.  Thus,
  as a dividing wall in $\Mov(X/Y)$, $W=\lambda^\perp$ extends to a
  hyperplane around $F$.
\end{proof}

One might expect that the subdivision of ${\Mov}(X/Y)$ by the flopping
clases can be related to the 2-dimensional components of the central
fiber $\pi^{-1}(0)$ of the 4-dimensional symplectic contraction $\pi:
X\ra Y$.  We note that, if $\pi':X'\ra Y$ is obtained from $X$ via a
Mukai flop $X -\ra X'$, then there is a natural bijection between
2-dimensional components of $\pi^{-1}(0)$ and of
$\pi'^{-1}(0)$. Indeed, a Mukai flop exchanges a number of $\bP^2$'s
into their "opposite" $\bP^2$'s, while on the other components it is a
composition of blow-ups and blow-downs. It is not clear however if,
after a sequence of Mukai flops, a given component of $\pi^{-1}(0)$
may become the locus of a different flopping classes.

\subsection{Essential curves}
The following definition of essential curves is a simplified version
of the one introduced in \cite{AltmannWisniewski}, suitable for the
present set-up.

\begin{definition} 
  Let $\pi: X\ra Y$ be a 4-dimensional local symplectic contraction
  with the unique 2-dimensional fiber $\pi^{-1}(0)$. Recall that by
  $N_1(X/Y)$ we denote the $\Q$-vector space of 1-cycles proper over
  $Y$. We define $\Ess(X/Y)$ as the convex cone spanned by the classes
  of curves which are not contained in $\pi^{-1}(0)$. Classes of
  curves in $\Ess(X/Y)$ are called {\it essential curves}.
\end{definition}

\begin{theorem}\label{dualityMovEss}
  {\rm (c.f.~\cite{AltmannWisniewski})} \label{Mov} The cones
  $\Mov(X/Y)$ and $\Ess(X/Y)$ are dual in terms of the intersection
  product of $N^1(X/Y)$ and $N_1(X/Y)$, that is $\Mov(X/Y) =
  \Ess(X/Y)^\vee$.
\end{theorem}
\begin{proof} First we note that if $D$ is a movable divisor on $X$,
  or if $|mD|$ has no fixed component for $m\gg 0$, then the base
  point locus of $|mD|$ is contained in $\pi^{-1}(0)$. This yields
the obvious inclusion $\Mov(X/Y)\subseteq\Ess(X/Y)^\vee$. Moreover,
since by \ref{4dimMDS} any two symplectic resolutions of $Y$ are
connected by a sequence of flops in centers over $0\in Y$, it follows
that the intersection of divisors with curves outside $\pi^{-1}(0)$
does not depend on the choice of the resolution (or SQM of $X$). That
is, if $C\subset X\setminus \pi^{-1}(0)$ is a curve proper over $Y$,
$\pi': X'\ra Y$ another symplectic resolution and $D'$, the
strict transform of $D$, then $D\cdot C=D'\cdot C$.
 
Now assume by contradiction that $\Mov(X/Y)\ne \Ess(X/Y)^\vee$.  Let
$F$ be a facet, i.e. a codimension 1 face of $\Mov(X/Y)$. Since $X$ is MDS,
we can take a resolution $\pi': X'\ra Y$, for which $F\cap\Nef(X'/Y)$
is an extremal face of $\Nef(X'/Y)$. The relative elementary
contraction $X'\ra Y'\ra Y$ of the face $F\cap\Nef(X'/Y)$ is
divisorial. Indeed, if it is not, then after a flop we would get
another $X''\ra Y$ whose nef cone $\Nef(X''/Y)$ is on the other side
of the face $F\cap\Nef(X'/Y)$, contradicting the fact that $F$ is an
extremal face of $\Mov(X/Y)$. Now we can choose a curve $C\subset
X'\setminus(\pi')^{-1}(0)=X\setminus\pi^{-1}(0)$ contracted by $X'\ra
Y'$ and we get the inclusion $F\subset C^\perp$. Thus every facet of
$\Mov(X/Y)$ is supported by a curve in $\Ess(X/Y)$, hence
$\Mov(X/Y)^\vee\subseteq\Ess(X/Y)$ and we are done.
\end{proof}

From the proof it follows that the above result remains true also if
$\pi: X\ra Y$ is a higher dimensional symplectic contraction and $X$
is MDS over $Y$, and essential curves are defined as those whose loci
are in codimension 1.

We note that, since the map $\pi$ is assumed to be projective, the cone
$\Nef(X/Y)$, hence also the cone $\Mov(X/Y)$, is of maximal
dimension. On the other hand we have the following observation.

\begin{proposition}\label{purely-divisorial}
  Let $\pi: X \ra Y$ be as in Theorem \ref{4dimMDS}
  (or, more generally, suppose that $X$ is a MDS over $Y$).\\
  The following conditions are equivalent.
  \begin{enumerate}
  \item the cone $\Ess(X/Y)$ is of maximal dimension,
  \item the cone $\Mov(X/Y)$ is strictly convex, that is it contains
    no linear subspace of positive dimension,
  \item the classes of components of fibers of $\pi$ outside
    $\pi^{-1}(0)$ generate $N_1(X/Y)$,
  \item the classes of exceptional divisors generate $N^1(X/Y)$.
\end{enumerate}
\end{proposition}

\begin{proof}
  In view of \ref{dualityMovEss} the equivalence of (1) and (2) is
  formal. Also (1) is equivalent to (3) by the definition of the cone
  $\Ess(X/Y)$. Finally, the intersection of classes of exceptional
  divisors with curves contained in general fibers of their
  contraction is a non-degenerate pairing,
  c.f.~\ref{root-systems}. Hence (3) is equivalent to (4).
\end{proof}

\section{Root systems and the structure of  $\Mov(X/Y)$.}
\subsection{Root systems}\label{root-systems}
We recall some generalities regarding root systems: a standard
reference for this part is \cite{BourbakiLie}. Consider a (finite
dimensional) real vector space $V$ with a euclidean product, root
lattice $\Lambda_R$ and weight lattice $\Lambda_W$, $\Lambda_W\supset\Lambda_R$. We
distinguish the set of simple (positive) roots denoted by $\{e_i\}$
and their opposite $E_i=-e_i$.  Note that the lattice $\Lambda_R$ is
spanned by $E_i$'s or $e_i$'s while its $\bZ$-dual is $\Lambda_W$. The
Cartan matrix describes the intersection $(e_i\cdot e_j)=-(e_i\cdot
E_j)$ which is also reflected in the respective Dynkin diagram.  Any
such root system is a (direct) sum of irreducible ones coming from the
infinite series $\bA_n, \bB_n, \bC_n, \bD_n$ and also $\bE_6,\bE_7,
\bE_8$ as well as $\bF_4$ and $\bG_2$. By abuse we denote by the respective letter both the Cartan matrix
and the associated root system.

The Cartan matrix of each of the systems $\bA_n$, $\bD_n$ and $\bE_6$,
$\bE_7$, $\bE_8$ has $2$ at the diagonal and $0$ or $-1$ outside the
diagonal. Given a group $H$ of automorphisms of any of the
$\bA-\bD-\bE$ Dynkin diagrams we can produce a matrix of intersections
of classes of orbits of the action. The entries are intersections of
an element of the orbit with the sum of all elements in the orbit,
that is: $(e_i\cdot\sum_{e_k\in H(e_j)}e_k)$. For example: the
involution identifying two short legs of the $\bD_n$ diagram
$\begin{xy}<10pt,0pt>: (-0.7,0.7)*={\bullet} ; (0,0)*={\bullet}="0"
  **@{-}, (-0.7,-0.7)*={\bullet} ; "0" **@{-}, "0" ;
  (1,0)*={\bullet}="1" **@{-},
\end{xy}\cdots$ described by the $n\times n$ Cartan matrix 
$$\left(\begin{array}{rrrrr}
2&0&-1&0&\cdots\\
0&2&-1&0&\cdots\\
-1&-1&2&-1&\cdots\\
0&0&-1&2&\cdots\\
\cdots&\cdots&\cdots&\cdots&
\end{array}\right)
$$
yields the $(n-1)\times(n-1)$ matrix associated with the system
$\bC_{n-1}$; we write $\bD_n/\bZ_2=\bC_{n-1}$:
$$\left(\begin{array}{rrrr}
2&-1&0&\cdots\\
-2&2&-1&\cdots\\
0&-1&2&\cdots\\
\cdots&\cdots&\cdots&
\end{array}\right)
$$

Similarly, we verify that $\bA_{2n+1}/\bZ_2=\bB_n$,
$\bD_{n}/\bZ_2=\bC_{n-1}$, $\bE_6/\bZ_2 = \bF_4$ and $\bD_4/\sigma_3 =
\bG_2$. The geometry behind these equalities is explained in Section
\ref{nilpotent-cone}.

Let ${\mathbb U}_n$ denote the following $n\times n$  matrix
\begin{eq}\label{U-matrix}
\left(\begin{array}{rrrrr}
1&-1&0&0&\cdots\\
-1&2&-1&0&\cdots\\
0&-1&2&-1&\cdots\\
0&0&-1&2&\cdots\\
\cdots&\cdots&\cdots&\cdots&\cdots
\end{array}\right)
\end{eq}
The matrix ${\mathbb U}_n$ is obtained from the root system $\bA_{2n}$
modulo involution of the respective Dynkin diagram. Here $\mathbb U$
stands for unreasonable (or un-necessary).

\subsection{The structure of $\Mov$ and $\Ess$}
The following is a combination of a result of Wierzba
\cite[1.3]{Wierzba} and Theorem \ref{dualityMovEss}.

\begin{theorem}\label{cone-structure}
  Let $\pi: X\ra Y$ be a local symplectic contraction (arbitrary
  dimension). Suppose that $N^1(X/Y)$ is generated by the classes of
  codimension 1 components of the exceptional set of $\pi$, we call them $E_\alpha$;
  that is we are in situation described in Proposition \ref{purely-divisorial}. Let
  $e_\alpha$ denote the numerical equivalence class of an irreducible
  component of a general fiber of $\pi_{| E_\alpha}$. Then the
  following holds:
  \begin{itemize}
  \item The classes of $E_\alpha$ are linearly independent so they
    form a basis of $N^1(X/Y)$.
  \item The opposite of the intersection matrix $-(e_\alpha\cdot
    E_\beta)$ is a direct sum of Cartan matrices of type associated with
    simple algebraic Lie groups (or algebras), and possibly, matrices
    of type ${\mathbb U}_n$.
  \item If moreover $X$ is MDS over $Y$ then $\Mov(X/Y)$ is dual, in
    terms of the intersection of $N^1(X/Y)$ and $N_1(X/Y)$, to the
    cone spanned by the classes of $e_\alpha$. In particular
    $\Mov(X/Y)$ is simplicial.
  \end{itemize}
\end{theorem}

In short, the above theorem says that, apart from the case $\mathbb
U_n$, which we do not expect to occur, the situation of an arbitrary local
symplectic contraction on the level of divisors and 1-cycles is very
much like in the case of the contraction to the nilpotent cone, which is
the case of Theorem \ref{BrieskornSlodowy} and Corollary
\ref{MovNefWeyl}.

\begin{conjecture}\label{un-necessary}
  The case ${\mathbb U}_n$ does not occur. That is, there is no
  symplectic contraction $X\ra Y$ such that $Y$ has a codimension 2 locus of
  $\bA_{2n}$ singularities and there exists a non-trivial
  numerical equivalence for curves in $X$ which are in a general fiber of $\pi$
  over this locus.
\end{conjecture}

Since in dimension four $X$ is an MDS over $Y$, in order
to prove this conjecture it is enough to deal with the case when $\pi: X\ra
Y$ is elementary and it is a contraction to $\bA_2$
singularities in codimension $2$. Indeed, take an irreducible curve $C_1$ whose
intersection with the irreducible divisor $E_1$ is $(-1)$; this is
the upper-right-hand corner of the matrix ${\mathbb U}_n$ in 
\ref{U-matrix}. The class of $C_1$ spans a ray on $\Ess(X/Y)$ and its
dual $C_1^\perp \cap \Mov(X/Y)$ is a facet of $\Mov(X/Y)$. Hence we
can choose an SQM model $X'$ with a facet of $\Nef(X'/Y)$ contained in
$C_1^\perp$. Thus there exists an elementary contraction of $X'$
which contracts $C_1$ whose exceptional locus is (the strict
transform of) $E_1$.

\begin{corollary} 
\label{Weyl}
  Suppose that the conjecture \ref{un-necessary} is true. Then, for
  every local symplectic contraction $\pi: X\ra Y$ satisfying the
  conditions of \ref{purely-divisorial}, there exists a semisimple Lie
  group and an identification of $N^1(X/Y)$ and $N_1(X/Y)$ with the
  real part of its Cartan algebra such that: (1) the intersection of
  the 1-cycles with classes of divisors is equal to the Killing form
  product, (2) the classes of irreducible essential curves spanning
  rays of $\Ess(X/Y)$ is identified with its primitive roots and (3)
  the cone $\Mov(X/Y)$ is identified with its Weyl chamber.
\end{corollary}

\begin{conjecture}
  Under the above identification the classes of (integral) 1-cycles
  should form the lattive $\Lambda_R$ of roots, while the classes of
  divisors should make the lattice $\Lambda_W$ of weights.
\end{conjecture}

\subsection{Examples of root systems}
We give a description of the cone of moving divisors and of the
related root system for three examples coming as resolutions of
$\bC^4/G$, where $G < Sp(\C^4) =: Sp(4)$ is a finite subgroup
preserving a symplectic form.  The reader unfamiliar with these
examples may prefer to read first Section \ref{examples}.

\begin{example} Let $BT$ be the binary tetrahedral group; that is $BT$
  is the preimage under the standard homomorphism $SU(2) \ra SO(3)$ of
  the subgroup of $SO(3)$ generated by the symmetries of a regular
  tetrahedron. This group has three irreducible representations on
  $\C^2$, the standard arising from the embedding into $SU(2)$ and two
  other one called $S_1$ and $S_2$.  The group $BT$ acts on the
  product $S_1\otimes S_2 = \C^4$. It is known that $\bC^4/BT$ has a
  symplectic resolution; in fact, recently Lehn and Sorger gave an
  explicit construction of it, see \cite{LehnSorger}.  This resolution
  is related to the root system $\bA_2$ with generators $e_1$ and
  $e_2$. Then $v=\pm(e_1-e_2)$ is the flopping system. Here is the
  picture of the weight lattice, together with roots denoted by
  $\bullet$ and flopping classes denoted by $\circ$. The $\Mov$ cone
  (or Weyl chamber) is divided into two parts by the line orthogonal
  to the flopping class.
\begin{eq}\label{C4/BT-diagram}
\begin{xy}<24pt,0pt>:
(-3,0)*={\circ} ; (3,0)*={\circ} **@{.},
(0,0)*={} ; (-1.5,2.598)*={} **@{-},
(0,0)*={} ; (1.5,2.598)*={} **@{-},
(0,0)*={} ; (0,2.598)*={} **@{-},
(0,0)*={} ; (-1,-1.732)*={} **@{.},
(0,0)*={} ; (1,-1.732)*={} **@{.},
(-1.5,0.866)*={\bullet} ; (1.5,0.866)*={\bullet} **@{.},
(-1.5,-0.866)*={\bullet} ; (1.5,-0.866)*={\bullet} **@{.},
(0,1.732)*={\bullet} ; (1.5,-0.866)*={} **@{.},
(0,-1.732)*={\bullet} ; (1.5,0.866)*={} **@{.},
(0,1.732)*={} ; (-1.5,-0.866)*={} **@{.},
(0,-1.732)*={} ; (1.5,0.866)*={} **@{.},
(-1.5,0.866)*={} ; (0,-1.732)*={} **@{.},
(-3,0.3)*={e_2-e_1}, (3,0.3)*={e_1-e_2},
(1.5,1.2)*={e_1}, (-1.5,1.2)*={e_2},
(-1.5,-1.2)*={E_1}, (1.5,-1.2)*={E_2},
(0.1,2.9)*={(e_1-e_2)^\perp}, (-1.7,2.7)*={e_1^\perp}, (1.8,2.7)*={e_2^\perp} 
\end{xy}
\end{eq}
\end{example}

\begin{example}
  Take the quotient $\bC^4/\bZ_{2}^2\rtimes\bZ_2$ and its Hilb-Chow
  resolution $X\ra Y$; this is Example \ref{HilbertChow} for $n=2$,
  more details can be found in Section \ref{examples}.  It is related
  to the decomposable root system $\bA_1\oplus\bA_1$ with roots
  denoted, respectively, by $e_0$ and $e_1$. The following picture
  describes a section of $\Mov$ together with its decomposition by
  flopping classes.
\begin{eq}\label{C4/Z2xZ2-diagram}
\begin{xy}<20pt,0pt>:
(-2,2)*={\circ} ; (2,2)*={} **@{.},
(-2,1)*={} ; (2,1)*={} **@{.},
(-2,0)*={\bullet} ; (2,0)*={\bullet} **@{.},
(-2,-1)*={} ; (2,-1)*={} **@{.},
(-2,-2)*={} ; (2,-2)*={\circ} **@{.},
(-2,2)*={} ; (-2,-2)*={} **@{.},
(-1,2)*={} ; (-1,-2)*={} **@{.},
(0,2)*={\bullet} ; (0,-2)*={\bullet} **@{.},
(2,2)*={} ; (2,-2)*={} **@{.},
(1,2)*={} ; (1,-2)*={} **@{.},
(0,0)*={} ; (3,0)*={} **@{-},
(0,0)*={} ; (3,3)*={} **@{-},
(0,0)*={} ; (0,3)*={} **@{-},
(-3,2)*={e_0-e_1}, (-0.3,2.2)*={e_0}, (2.3,0.3)*={e_1}, (3,-2)*={e_1-e_0},
(3.3,2)*={(e_0-e_1)^\perp}
\end{xy}
\end{eq}

\end{example}

\begin{example}
  Take the quotient $Y=\bC^4/\bZ_{3}^2 \rtimes \bZ_2$ and its
  Hilb-Chow resolution $X\ra Y$; this is Example \ref{HilbertChow} for
  $n=3$.  It is
  related to the decomposable root system $\bA_2\oplus\bA_1$ with
  roots denoted, respectively, by $e_1,\ e_2$ and $e_0$. The following
  picture describes a plane section of a 3-dimensional cone
  $\Mov(X/Y)$ (denoted by solid line segments) together with its
  decomposition by flopping classes (denoted by dotted line
  segments). The upper chamber in this picture is the nef cone
  $\Nef(X/Y)$. This situation will be discussed in detail in
  \ref{A_2+A_1} and \ref{A_n+A_1}.
\begin{eq}\label{C4/Z3xZ2-diagram}
\begin{xy}<14pt,0pt>:
(-4,0)*={} ; (4,0)*={} **@{-},
(-4,0)*={} ; (0,-8)*={} **@{-}, 
(4,0)*={} ; (0,-8)*={} **@{-}, 
(-4,-4)*={} ; (7,-4)*={} **@{.},
(-5,0.666)*={} ; (3,-4.666)*={} **@{.}, 
(5,0.666)*={} ; (-3,-4.666)*={} **@{.},
(0,0.5)*={e_0^\perp}, (-3,-3)*={e_1^\perp}, (1.5,-6)*={e_2^\perp},
(6,-4.2)*={(e_0-e_1-e_2)^\perp}, (-1.5,-1)*={(e_0-e_1)^\perp}, 
(1.5,-1.6)*={(e_0-e_2)^\perp}
\end{xy}
\end{eq}

\end{example}


\section{Rational curves and differential forms}\label{rational-curves}
\subsection{The set-up}
Let $\pi: X \ra Y$ be a local symplectic contraction of a 4-fold. We
assume that we are in the situation of \ref{purely-divisorial}. In
particular, the exceptional locus of $\pi$ is a divisor $D$. This
divisor, as well as its image surface $S:=\pi (D) \subset Y$, can be
reducible. As above $0 \in S \subset Y$ denotes the unique point over
which $\pi$ can have a two dimensional fiber.

Our starting point is the paper of Wierzba \cite{Wierzba} (as well as
the appendix of \cite{SolaCondeWisniewski}) to which we will often refer. In
particular Theorem 1.3 of \cite{Wierzba} says that a general fiber of
$\pi$ over any component of $S$ is a configuration of $\pu$'s with
dual graph being a Dynkin diagram.  The components of these fibers are
called essential curves in the previous section.

Choose an irreducible component of $S$, call it $S'$. Take an
irreducible curve $C\iso\bP^1$ in a (general) fiber over a point in
$S'\setminus\{0\}$ and let $D'$ be the irreducible component of $D$
which contains $C$; note that $\pi(D') = S'$ and $S'$ may be (and
usually is) non-normal.  Let $\cV' \subset Chow(X/Y)$ be an
irreducible component of the Chow scheme of $X$ containing $C$. By
$\cV$ we denote its normalization and $p: \cU \ra \cV$ is the
normalized pullback of the universal family over $\cV'$. Finally, let
$q : \cU \ra D' \subset X$ be the evaluation map, see
e.g.~\cite[I.3]{Kollar} for the construction. The contraction $\pi$
determines a morphism $\tilde\pi: \cV \ra S'$, which is surjective
because $C$ was chosen in a general fiber over $S'$. We let $\mu :\cV
\ra \widetilde S' \ra S'$ be its Stein factorization. In particular
$\widetilde S'$ is normal and $\nu:\widetilde S' \ra S'$ is a finite
morphism, \`etale outside $\nu^{-1}(0)$, whose fibers are related to
the orbits of the action of the group of automorphism of the Dynkin
diagram, \cite[1.3]{Wierzba}. We will assume that $\mu$ is not an
isomorphism which is equivalent to say that $D'$ has a 2-dimensional
fiber over $0$.
Also, since we are interested in understanding the local description
of the contraction in analytic category we will assume that $S'$ is
analytically irreducible at $0$ or that $\nu^{-1}(0)$ consists of
single point. The exceptional locus of $\mu$ is $\mu^{-1}(\nu^{-1}(0))
= \bigcup_i V_i$ where $V_i\subset\cV$ are irreducible curves.
\begin{eq}\label{evaluation-diagram}
\xymatrix{
\cU  \ar[d]_{p}  \ar[rr]^(.40){q}& &D' \ar[d]_{\pi}  \subset X \\
\cV \ar[r]^(.40)\mu &\widetilde S'  \ar[r]^(.40)\nu&S' \subset Y}
\end{eq}

If necessary, we can take $\cV$ to be smooth, eventually by replacing
it with its desingularization and $\cU$ with the normalized fiber
product.  A general fiber of $p:\cU\ra\cV$ is $\bP^1$ while other
fibers are, possibly, trees of rational curves. If $C$ is an extremal
curve, which by \ref{dualityMovEss} and \ref{4dimMDS} is true for some
SQM model of $X$, then $-D$ is ample on the extremal ray spanned by
$C$ and since $-D\cdot C\leq 2$ it follows that $p:\cU\ra\cV$ is a
$\bP^1$ or conic bundle. Since any two SQM models of $X$ are obtained
by a sequence of Mukai flops, it follows that in a general situation
$p:\cU\ra\cV$ is obtained by a sequence of blows and blow-downs of a
$\bP^1$ or conic bundle.

In \cite{Wierzba} and \cite{SolaCondeWisniewski} it was proved that
$S' \setminus \{ 0 \}$ is smooth and that, on $\cV \setminus \{ (\nu
\circ\mu)^{-1}(0)\}$, $p$ is a $\bP^1$-bundle. It was also showed, by
pulling back the symplectic form via $q$ and pushing it further down
via $p$, that one can obtain a symplectic form on $S' \setminus \{
0\}$. We will repeat their procedure in this more general case.

\subsection{The differentials}\label{differentials}
Let us consider the derivative map $Dq: q^*\Omega _X \ra
\Omega_{\cU}$. Its cokernel is a torsion sheaf, call it
$\cQ_{\Delta_2}$, supported on the set $\Delta_2$, which is the set of
points where $q$ is not of maximal rank: by the theorem on the purity
of the branch locus $\Delta_2$ is a divisor. As for the kernel, let
$I$ be the ideal of $D'$ in $X$ and consider the sequence $q^*(I/I^2)
\ra q^* \Omega_X \ra \Omega_{\cU}$. The saturation of the image of the
first map will be the kernel of the second map and it will be a
reflexive sheaf of the form $\cO_\cU (-D' + \Delta_1)$, with
$\Delta_1$ being an effective divisor. In the above notation we can
write the exact sequence

\begin{eq}\label{diff1}
  0 \lra \cO_\cU(q^*(-D') + \Delta_1) \lra q^* \Omega_X \lra
  \Omega_{\cU} \lra \cQ_{\Delta_2} \lra 0.
\end{eq}

We have another derivation map into $\Omega_{\cU}$, namely $Dp:
p^*\Omega _{\cV} \ra \Omega_{\cU}$.  It fits in the exact sequence

\begin{eq}\label{diff2}
  p^*\Omega _{\cV} \lra \Omega_{\cU} \lra \Omega_{\cU/\cV} \lra 0,
\end{eq}

whose dual sequence is 
\begin{eq}\label{diff3}
  0 \lra T_{\cU/\cV} \lra T_{\cU} \lra p^*T _{\cV}
\end{eq}

The symplectic form on $X$, that is $\omega_X$, gives an isomorphism
$\omega_X: T_X \ra \Omega_X$. We consider the following diagram
involving morphism of sheaves over $\cU$ appearing in the above
sequences.
\begin{eq}\label{main-diagram}
  \xymatrix{T_{\cU/\cV}\ar[r]&T_{\cU}\ar[d]^{(Dq)^*}\ar[r]^{(Dp)^*}&
    p^*(T_{\cV})\ar@{.>}@/^/[r]^{p^*(\omega_\cV)}&
    p^*(\Omega_{\cV})\ar[d]^{Dp}\\
    &q^*T_X \ar[r]_{q^*(\omega_X)}&q^*\Omega_X \ar[r]^{Dq}
    &\Omega_{\cU}\ar[r]&\Omega_{\cU/\cV}}
\end{eq}
We claim that the dotted arrow exists and it is obtained by a pull
back of a two form $\omega_\cV$ on $\cV$, and it is an isomorphism
outside the exceptional set of $\mu$ which is $\bigcup_i V_i$. Indeed,
the composition of arrows in the diagram which yields
$T_\cU\ra\Omega_\cU$ is given by the 2-form $Dq(\omega_X)$ and it is
zero on $T_{\cU/\cV}$, because this is a torsion free sheaf and its
restriction to any fiber of $p$ outside $\bigcup_i V_i$ (any fiber of
$p$ is there a $\bP^1$) is $\cO(2)$ while the restriction of
$\Omega_\cU$ is $\cO(-2)\oplus\cO\oplus\cO$.  By the same reason the
composition $T_\cU\ra\Omega_\cU \ra \Omega_{\cU/\cV}$ is zero since
$T_{\cU}$ on any fiber of $p$ outside $\bigcup_i V_i$ is
$\cO(2)\oplus\cO\oplus\cO$ while $\Omega_{\cU/\cV}$ is $\cO(-2)$. Thus
the map $Dq(\omega_X): T_\cU\ra\Omega_\cU$ factors through
$p^*(T_\cV)\ra p^*(\Omega_\cV)$ and, as a result,
$Dq(\omega_X)=Dp(\omega_\cV)$, for some $2$-form $\omega_\cV$ on
$\cV$. Since $Dq$ is of maximal rank outside of $p^{-1}(\bigcup_i
V_i)$ and $p$ is just a $\bP^1$-bundle there, it follows that
$\omega_\cV$ does not assume zero outside the exceptional set of
$\mu$. Hence $K_{\cV} = \sum a_iV_i$, with $a_i \geq 0$ being the
discrepancy of $V_i$. Note that the above argument follows essentially
from \cite[Sect.~4.1]{SolaCondeWisniewski} or \cite[Sect.~5]{Wierzba}.

\begin{theorem}\label{DuValsing}
  The surface $\widetilde S'$ has at most Du Val (or $\bA-\bD-\bE$)
  singularity at $\nu^{-1}(0)$ and $\mu:\cV\ra\widetilde S'$ is its,
  possibly non-minimal,
  resolution. In particular every $V_i$ is a
  rational curve. If a component $V_i$ has positive discrepancy or,
  equivalently, the form $\omega_\cV$ vanishes along $V_i$, then
  $p^{-1}(V_i) \subset \Delta_2$.
\end{theorem}
\begin{proof}
  The first statement follows from the discussion preceeding the
  theorem. To get the second part, note that over $\cU$ we have
  $Dq(\omega_X) = Dp(\omega_\cV)$ and $\omega_\cV$ is zero at any
  component of $\bigcup_iV_i$ of positive discrepancy. Since
  $\omega_X$ is nondegenerate this equality implies that $Dq$ is of
  rank $\leq 2$ on the respective component of $p^{-1}(\bigcup V_i)$.
\end{proof}

We note that although the surface $\widetilde S'$ is the same for all
the symplectic resolutions of $Y$, the parametric scheme for lines,
which is a resolution of $\widetilde S'$ may be different for
different SQM models, see \ref{changeV} for an explicit example.

\begin{proposition}\label{discrepancy}
  Suppose that the map $p$ is of maximal rank in codimension 1. Then
  the $p$-inverse image of the set of positive discrepancy components
  of $\bigcup_iV_i$ coincides with the set where the rank of $q$
  drops. That is, $\Delta_2$ is the pullback of the zero set of
  $\omega_\cV$.
\end{proposition}

\begin{proof}
  We have the following injective morphism of sheaves
  $q^*(\omega_X)\circ (Dq)^*(T_{\cU/\cV}) \hookrightarrow
  \cO_\cU(-p^*D+\Delta_1)\hookrightarrow q^*\Omega_X $ which follows,
  as already noted, because of the splitting type of $\Omega_\cU$. We
  claim that this implies the isomorphism of line bundles
  $T_{\cU/\cV}\iso \cO_\cU(-p^*D+\Delta_1)$. Indeed, the evaluation
  map of the universal family over the Chow scheme is isomorphic on
  the fibers, hence $(Dq)^*$ is of maximal rank along $T_{\cU/\cV}$ is
  codimension 1 at least, hence the desired isomorphism. 

  Now, since $p$ is submersive in codimension 1, because of the
  sequence \ref{diff2} we can write $det\Omega_\cU = p^*(K_\cV)
  \otimes \Omega_{\cU/\cV}$ and consequently, because of the sequence
  \ref{diff1}, we get
  $$c_1(\cQ_{\Delta_2})=c_1(\cO_\cU(-p^*D+\Delta_1))- 
  c_1(T_{\cU/\cV})+c_1(p^*(K_\cV))=c_1(p^*K_{\cV}) = p^*(\sum
  a_i[V_i])$$
\end{proof}

\subsection{Vertical slicing}
The first of the following two results is essentially known,
c.f.~\cite[2.3]{KaledinPoisson} and also \cite[1.2(ii),
1.4.]{Wierzba}. We restate and reprove it in the form suitable for
the subsequent corollary.
\begin{proposition}\label{symplectic-strata}
  Suppose that $\pi: X\ra Y$ is a symplectic contraction with $\dim
  X=2n$.  Let $Z\subset X$, with $\codim Z=m$, be a (irreducible)
  maximal cycle with $S=\pi(Z)$, $\codim S=2m$. The fibers of
  $\pi_{| Z}: Z\ra S$ are isotropic (with respect to $\omega_X$) and,
  moreover, over an open and dense set $S_0\subset S$ there exists a
  symplectic form $\omega_S$ such that over $\pi_{|Z}^{-1}(S_0)$ we
  have $D\pi(\omega_S)=\omega_{X|Z}$
\end{proposition}
\begin{proof}
  The proof that $\omega_X$ restricted to fibers of $\pi$ is zero, so
  that they are isotropic (or lagrangian), is in
  \cite[2.20]{WierzbaWisniewski}. Let $\iota: Z\ra X$ be the
  embedding. Then we have the following version of diagram
  \ref{main-diagram}
\begin{eq}\label{main-diagram-bis}
  \xymatrix{T_{Z/S}\ar[r]&T_{Z}\ar[d]^{(D\iota)^*}\ar[r]^{(D\pi)^*}&
    \pi^*(T_{S})\ar@{.>}@/^/[r]^{\pi^*(\omega_S)}&
    \pi^*(\Omega_{S})\ar[d]^{D\pi}\\
    &\iota^*T_X \ar[r]_{\iota^*(\omega_X)}&\iota^*\Omega_X
    \ar[r]^{D\iota} &\Omega_{Z}\ar[r]&\Omega_{Z/S}}
\end{eq}
We claim the existence of $\omega_S$. The composition
$T_{Z/S}\ra\Omega_{Z/S}$ is trivial since fibers of $\pi$ are
isotropic. On the other hand the induced maps
$\pi^*(T_S)\ra\Omega_{Z/S}$ and $T_{Z/S}\ra \pi^*\Omega_S$ are zero:
indeed, otherwise we would have nonzero 1 forms on a generic fiber of
$\pi_{|Z}$, which would contribute to the first cohomology of the
fiber (via the Hodge theory on the simplicial resolution of the fiber)
which contradicts \cite[2.12]{KaledinPoisson}.

Thus the dotted arrow in the above diagram is well defined and it
satisfies $D\pi(\omega_S) = D\iota(\omega_X)$ for a two form
$\omega_S$ defined over a smooth subset $S_0$ of $S$. Moreover the
form $\omega_S$ is of maximal rank for the dimensional reasons.
\end{proof}

The following corollary is a symplectic version of
\cite[1.3]{AndreattaWisniewski}. 

\begin{corollary}\label{vertical-slicing}
  {\rm[Vertical slicing]} In the situation of \ref{symplectic-strata}
  let $H_1,\dots H_{2n-2m}$ be general irreducible divisors in $Y$
  meeting in a general point $s\in S$.  Letting $Y'=H_1\cap\cdots\cap
  H_{2n-2m}$ and $X'=\pi^{-1}(Y')$, possibly shrinking $Y'$  
  to a neighborhood of $s$, we get that $\pi'=\pi_{|X'}:X'\ra Y'$ is a
  local symplectic contraction of a $2m$-fold with an exceptional fiber,
  $\pi^{-1}(s)$, of dimension $m$.
\end{corollary}

\begin{proof}
  Since $\pi$ is crepant it is enough to show that the restriction of
  $\omega_X$ to $X'$ is nondegenerate at a point over $s$ in order to
  claim that it is symplectic over the whole $X'$ (after possibly
  shrinking $Y'$ to a neighborhood of $s$). To this end we consider
  the following commuting diagram with exact rows
\begin{eq}\label{slicing-diagram}
  \xymatrix{
    0\ar[r]&T_{X'|Z_s}\ar[d]^{\omega_{X|X'}}\ar[r]&
    T_{X|Z_s}\ar[d]^{\omega_X}\ar[r]&
    (N_{X'/X})_{Z_s}=(T_S)_s\otimes\cO_{Z_s}\ar[d]^{D\pi(\omega_S)}\ar[r]&0\\
    0&\Omega_{X'|Z_s}\ar[l]&\Omega_{X|Z_s}\ar[l]&
(N^*_{X'/X})_{|Z_s}=(T_S^*)_s\otimes\cO_{Z_s}\ar[l]&0\ar[l]
}
\end{eq}
Here $Z_s=Z\cap X'$ is a complete intersection,
hence$(N_{X'/X})_{|Z_s} = N_{Z_s/Z}$ which yields the identifications
in the last non-zero column of the diagram. The right-hand-side
vertical arrow follows because of \ref{symplectic-strata}, where we
have also shown that is an isomorphism. This implies that the
left-hand-side vertical arrow is an isomorphism too.
\end{proof}

\section{Quotient symplectic singularities}
\label{examples}

\subsection{Preliminaries} In this section $G < Sp(\C^4) =: Sp(4)$ is
a finite subgroup preserving a symplectic form. We will discuss some
examples in which $Y:= \C^4 /G$ admits a symplectic resolution $\pi:
X\ra Y$. We have the following two fundamental results about such
resolutions, the latter one known as the McKay correspondence.
\begin{theorem} {\rm(c.f. \cite{Verbitsky})} If $Y$ admits a
  symplectic resolution then $G$ is generated by symplectic
  reflections, that is elements whose fixed points set is of
  codimension 2.
\end{theorem}
\begin{theorem}\label{McKay}
  {\rm (c.f. \cite{KaledinMcKay})} The homology classes of the maximal
  cycles (as defined in \ref{semis}) form a basis of rational homology
  of $X$ and they are in bijection with conjugacy classes of elements
  of $G$.
\end{theorem}

On the other hand we have the following immediate observation (for
further details see for instance section 3.2 in
\cite{AndreattaWisniewski2}).

\begin{lemma}\label{singS}
  Let $S'\subset Y$ be a component of the codimension 2 singular locus
  associated with the isotropy group $H<G$. Then $H$ is one of the
  $\bA-\bD-\bE$ groups (a finite subgroup of $SL(\bC^2)$) consisting
  of symplectic reflections. The normalization of $S'$ has a
  quotient singularity by the action of $W(H)=N_G(H)/H$, where
  $N_G(H)$ is the normalizer of $H$ in $G$.
\end{lemma}

\subsection{Direct product resolution} Let $H_1, H_2 < SL(2)$ be
finite subgroups and consider $G:= H_1 \times H_2$ acting on $\C^4 =
\C^2 \times \C^2$. Let $\pi _i: S_i \ra \C^2/H_i$ be minimal
resolutions and $n_i=|H_i|-1$ be the number of exceptional rational
curves in $S_i$. The product morphism $\pi=\pi _1 \times \pi _2 : X :=
S_1 \times S_2 \ra Y:= \C^4/G$ is a symplectic resolution with the
central fiber isomorphic to the product of the exceptional loci of
$\pi_i$. In particular $X$ does not admit any flop and
$\Mov(X/Y)=\Nef(X/Y)$.  Every component of $Chow(X/Y)$ containing an
exceptional curve of $\pi_i$ is isomorphic to $S_j$, with $i \not= j
\in \{1,2\}$.

\subsection{Elementary contraction to  $\bC^4/\sigma_3$}
\label {sigma3a}
Let $\sigma_3$ be a group of permutation of $3$ elements; $\sigma_3$
acts on $\C^2$ via the standard representation.  Let $\sigma_3$ acts
on $\C^4 = \C^2 \oplus \C^2$ as the diagonal action of the standard
representation. A symplectic resolution of the quotient
$\bC^4/\sigma_3$ can be obtained as a section of the Hilbert-Chow
morphism $\tau: Hilb^3(\bC^2)\ra (\bC^2)^{(3)}$ in \ref{HilbertChow}.
This is a local version of Beauville's construction,
\cite{BeauvilleKahlerVar} (see also \cite{AndreattaWisniewski2}), and
it can be seen as a special case of \ref{vertical-slicing}. More
explicit calculations on this resolution are done in section
\ref{sigma3}.

There are three conjugacy classes in
$\sigma_3$ which are related to three maximal cycles, of complex
dimension 4, 3 and 2, each related to a $1$-dimensional group of
homology for the resolution $\pi: X\ra Y=\bC^4/\sigma_3$.

Since the normalizer of any order 2 element in $\sigma_3$ (any
transposition or any reflection, if one thinks about $\sigma_3$ as the
dihedral group) is trivial, by Lemma \ref{singS} it follows that the
normalization of the singular locus $S$ of $Y$ is smooth. Hence, by
\ref{discrepancy} we can compute both the parametrizing scheme for
rational curves in $X$ and the respective universal family. That is,
the parametrizing scheme $\cV$ is just a blow-up of the normalization
of $S$, the evaluation map $q: \cU \ra X$ drops its rank over $0$ and
the exceptional divisor of $\pi$, which is the image of $q$, is
non-normal over $0$.

\subsection{Wreath product} \label{wreath-product}
Let $H < SL(2)$ be a finite subgroup and let $G := H^{\times 2}\rtimes
\Z_2$ where $\bZ_2$ interchanges the factors in the product. We write
$G=H \wr \Z_2$.  Note that $\Z_{n+1}\wr\bZ_2$ has another nice
presentation, namely $(\Z_{n+1})^{\times 2}\rtimes \Z_2= D_{2n}
\rtimes \Z_n$, where  $D_{2n}$ is the dihedral group of
the regular $n$-gon and $\Z_n$ acts on it by rotations.

We consider the projective symplectic resolution described in
\ref{HilbertChow} (with $n=2$):
$$\pi: X:= Hilb^2(S) \ra S^{(2)} \ra(\C^2/H)^{(2)}:=Y$$
where $\nu: S\ra \bC^2/H$ is the minimal resolution with the
exceptional set $\bigcup_iC_i$, where $C_i$, $i=1, ..., k$, are
$(-2)$-curves.  

The morphism $\tau: Hilb^2(S) \ra S^{(2)}$ is just a blow-up of the locus of
$\bA_1$ singularities (the image of the diagonal under $S^2\ra
S^{(2)}$) with irreducible exceptional divisor $E_0$ which is a
$\bP^1$ bundle over $S$. We set $S'=\pi(E_0)$.  By $E_i$, with
$i=1,\dots, k$ we denote the strict transform, via $\tau$, of the
image of $C_i \times S$ under the map $S^2 \ra S^{(2)}$.  By $e_i$ we
denote the class of an irreducible component of a general fiber of
$\pi_{|E_i}$. The image $\pi(E_i)$ for $i \geq 1$ is the surface $S''
\iso C^2/H$.  The singular locus of $Y$ is the union $S = S' \cup
S''$.

The irreducible components of $\pi^{-1} (0)$ are described in the following.
\begin{itemize}
\item $P_{i,i}$, for $i = 1,..., k$. They are the strict transform of
  $C_i^{(2)}$ via $\tau$. They are isomorphic to $\proj^2$.
\item $P_{i,j}$, for $i,j = 1, ..., k$ and $i < j$. They are the
  strict transform via $\tau$ of the image of $C_i \times C_j$ under
  the morphism $S^2 \ra S^{(2)}$. They are isomorphic to $\proj^1
  \times \proj^1$ if $C_i \cap C_j = \emptyset$ and to the blow up of
  $\proj^1 \times \proj^1$ if $C_i \cap C_j \ne \emptyset$.
\item $Q_i$, for $i = 1,..., k$. They are the preimage
  $\tau^{-1}(\Delta_{C_i})$, where $\Delta_{C_i}$ is the diagonal
  embedding of $C_i$ in $S^{(2)}$. They are isomorphic to
  $\proj(T_{S|C_i}) = \proj(\cO_{C_i} (2) \oplus \cO_{C_i} (-2) )$, i.e. to the
  Hirzebruch surface $F_4$.
\end{itemize}

\smallskip Let us also describe some intersections between
these components. Namely, $P_{i,i}$ intersects $Q_i$ along a curve
which is a $(-4)$-curve in $Q_i$ and a conic in $P_{i,i}$.  If $C_i \cap
C_j = \{x_i\}$ then $P_{i,j}$ intersect $P_{i,i}$ (respectively
$P_{j,j}$) along a curve which is a $(-1)$ curve in $P_{i,j}$ and a line
in $P_{i,i}$ (respectively in $P_{j,j}$).  Moreover in this case
$P_{i,j}$ intersect $Q_i$ (respectively $Q_j$) in a curve which is a
$(-1)$ curve in $P_{i,j}$ and a fiber in $Q_i$ (respectively $Q_j$).

The next lemma is straightforward, a proof of it can be found in
\cite[Lemma 4.2]{FuSurvey}.

\begin{lemma} 
\label{qi}
The strict transform of $Q_i$ under any sequence of Mukai flops along
components in $\pi^{-1}(0)$ is not isomorphic to $\proj ^2$.
\end{lemma}

\subsection{Resolutions of $\bC^4/(\bZ_3\wr\bZ_2)$}\label{A_2+A_1}
The Figure 1 presents a ``realistic'' description of configurations of
components in the special fiber of symplectic resolutions of
$\bC^4/(\bZ_3\wr\bZ_2)$.  By abuse, the strict transforms of the
components and the results of the flopping of $\bP^2$'s are denoted by
the same letters.
\begin{figure}[h]
\includegraphics[width=12cm]{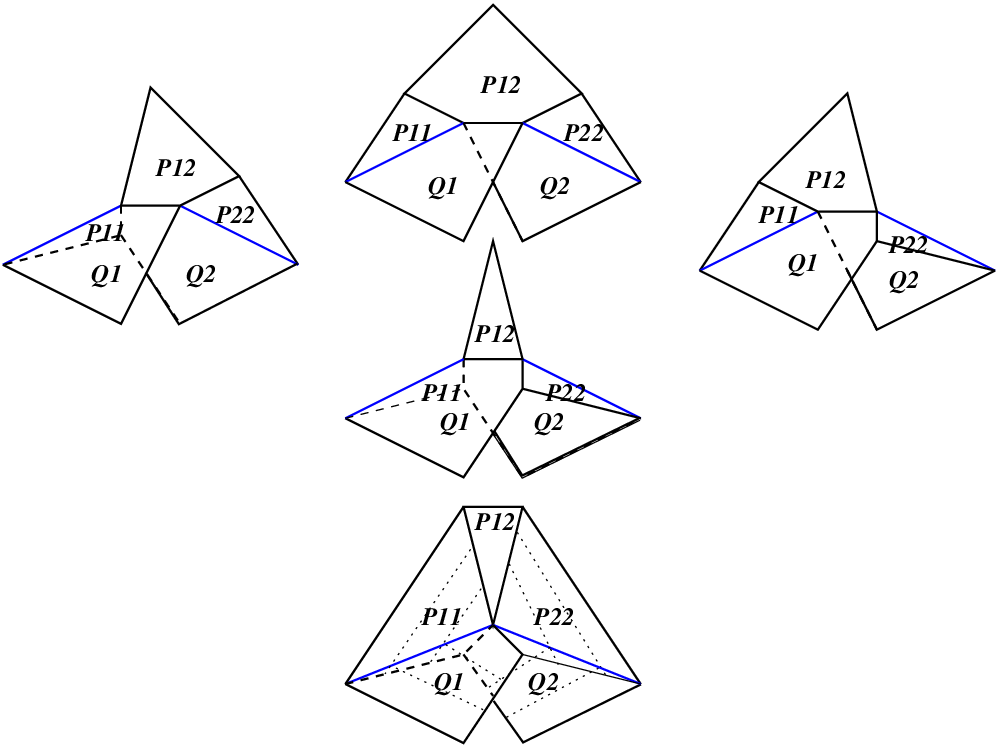}
\caption{Components of the central fiber in resolutions
  of $\bC^4/(\bZ_3\wr\bZ_2)$}
\end{figure}

The position of these configurations in Figure 1 is consistent with
the decomposition of the cone $\Mov(X/Y)$ presented in Diagram
\ref{C4/Z3xZ2-diagram}. In particular, the configuration at the top is
associated with the Hilbert-Chow resolution.  Note that the central
configuration of this diagram cointains three copies of $\bP^2$,
denoted $P_{ij}$, which contain lines whose classes are $e_0-e_1$,
$e_0-e_2$ and $e_1+e_2-e_0$.

On the other hand, the configuration in the bottom is associated with
the resolution which can be factored by two different divisorial
elementary contractions of classes $e_1$ and $e_2$. In fact,
contracting both $e_1$ and $e_2$ is a resolution of $\bA_2$
singularities which is a part of a resolution of $Y$ which comes from
presenting $\bZ_3\wr\bZ_2$ as $D_6\rtimes\bZ_3$.
That is, $X$ is then obtained by first resolving the singularities of
the action of $D_6=\sigma_3$ and then by resolving the singularities of
the $\bZ_3$ action on this resolution. 
The rulings of the respective surfaces coming from this last blow up
are indicated by dotted line segments. 
We will call such $X$ a $D_6\rtimes\bZ_3$-resolution.

This example is convenient for understanding the contents of
Theorem \ref{DuValsing} and of Proposition \ref{discrepancy}. We
refer to diagram \ref{evaluation-diagram} and let $S'$ and $S''$ be
the closure of the locus of $\bA_1$ and $\bA_2$ singularities in
$Y=\bC^4/(\bZ_3\wr\bZ_2)$. From Lemma \ref{singS} we find out that
the normalization of $S'$ as well as $S''$ has a singularity of type $\bA_2$.

By $\cV_0$ we denote the component of $Chow(X/Y)$ dominating $S'$ and
parametrizing curves equivalent to $e_0$, while by $\cV_1$ and $\cV_2$
we denote components dominating $S''$ parameterizing deformations of
$e_1$ and $e_2$. The surfaces $\cV_i$ may depend on the resolution
and, in fact, while $\cV_1$ and $\cV_2$ remain unchanged, the
component $\cV_0$ will change under flops.

\begin{lemma}\label{changeV}
  If $X$ is the Hilbert-Chow resolution then $\cV_0$ is the minimal
  resolution of $\bA_2$ singularity. If $X$ is the
  $D_6\rtimes\bZ_3$-resolution then $\cV_0$ is non-minimal, with one
  $(-1)$ curve in the central position of three exceptional curves.
\end{lemma}
\begin{proof}
  The first statement is immediate. To see the second one, note that
  we have the map of $\cV_0$ to Chow of lines in the resolution of
  $\bC^4/\sigma_3$ divided by $\bZ_3$ action. The $\bZ_3$-action in
  question is just a lift up of the original linear action on the
  fixed point set of rotations in $\sigma_3=D_6$ hence $\cV_0$
  resolves 2 cubic cone singularities associated with the eigenvectors of
  the original action.
\end{proof}

One may verify that in the $D_6\rtimes\bZ_3$-resolution the
exceptional set in $\cV_0$ 
parametrizes curves consisting of three components: $\Q_2\cap P_{11}$,
$Q_1\cap P_{22}$ and a line in $P_{12}$, whose classes are,
respectively, $e_2$, $e_1$ and $e_0-(e_1+e_2)$.

\subsection{Resolutions of  $\bC^4/(\bZ_{n+1}\wr\bZ_2)$}\label{A_n+A_1}
We use the notation introduced in Theorem \ref{cone-structure},
Corollary \ref{Weyl} and Theorem \ref{hyperplane}. Suppose to be in the set up of \ref{wreath-product} 
with $H=\bZ_{n+1}$. In particular the classes
$e_i$, for $i=1,\dots, n$, are identified to simple roots associated with consecutive nodes
of the Dynkin diagram $\bA_n$.

\begin{theorem}\label{cone-division}
  Let $X\rightarrow Y=\C^4/(\Z_{n+1}\wr \Z_2)$ be a
  symplectic resolution as above.  The division of $\Mov(X/Y)=\langle
  e_0,\dots , e_n\rangle ^\vee$ into Mori chambers is defined by
  hyperplanes $\lambda_{ij}^\perp$ for $1\leq i\leq j\leq n$, where
  $\lambda_{ij}=e_0-(e_i+e_{i+1}+\cdots+e_{j-1}+e_j).$
  
 \end{theorem}

 A proof of this theorem will occupy the rest of this section. 

 We know one Mori chamber of $\Mov(X/Y)$, the one associated with the
 Hilbert-Chow resolution. The faces of this chamber are supported by
 $e_0^\perp$ and by $-\lambda_{ii}^\perp=(e_i-e_0)^\perp$, see
 e.g. the above discussion. Thus, in particular, if $\lambda\in
 N_1(X)$ is a flopping class then $\lambda^\perp$ does not meet the
 relative interior of the face $\Mov(X)\cap e_0^\perp$.

 On the other hand, $\Mov(X/Y)=\Mov(X/Y)\cap e_0^\perp+\bR_{\geq 0}\cdot
 (-E_0)$. Thus, if we take any $D_0$ in the relative interior of
 $\Mov(X)\cap e_0^\perp$ then, by the above observation, the half-line
 $D_0+\bR_{\geq 0}\cdot(-E_0)$ must meet the hyperplane $\lambda^\perp$,
 for any flopping class $\lambda$. Hence the theorem will be proved if,
 for a choice of $D_0$, we will show that all the hyperplanes
 $\lambda^\perp$ that $D_0+\bR_{\geq 0}\cdot(-E_0)$ meets actually
 come from the classes $\lambda_{ij}$.

 Let us choose a sequence (a vector) of $n$ positive numbers
 $\overline{\beta}=(\beta_i)$ such that $\beta_1+\cdots+\beta_{i-1}<
 \beta_i$, for $i=2,\dots,n$. We set $\gamma_{ij}= \beta_i +
 \cdots+\beta_j$. Then, by our assumption,
\begin{eq}\label{gammas}
  \gamma_{11}<\gamma_{22}<\gamma_{12}<\gamma_{33}<\gamma_{23}<
  \gamma_{13}<\gamma_{44}<\gamma_{34}\cdots 
\end{eq}

Let $A$ be the intersection matrix for the root system $\bA_n$. The
matrix $-A$ is negative definite, therefore there exists a unique
vector $\overline{\alpha}=(\alpha_i)$ such that
$(-A)\cdot\overline{\alpha} = \overline{\beta}$. If we now set
$D_0=\sum_i\alpha_iE_i$ then $D_0\cdot e_0=0$ and $D_0\cdot
e_i=\beta_i>0$ for $i=1,\dots,n$; hence $D_0$ is in the relative interior
of $\Mov(X)\cap e_0^\perp$. What is more, if we set $D_t=D_0-(t/2)E_0$
then $D_t\cdot\lambda_{ij} = t-\gamma_{ij}$; so that $\gamma_{ij}$ is
the threshold value of $t$ for the form $\lambda_{ij}$ on the
half-line $\{D_t:t\in\bR_{\geq 0}\}$. The SQM model of $X$ on which
the divisor $D_t$ is ample will be denoted by $X_t$.

Now our theorem is equivalent to saying that the models $X_t$ are in
bijection with connected components (open intervals) in
$\bR_{>0}\setminus\{\gamma_{ij}\}$. This can be verified by starting
from the Hilb-Chow resolution associated with the interval $(0,\gamma_{11})$ and
proceeding inductively as it follows.  Let $t$ be in the interval
$(\gamma_{ij}, \gamma_{i'j'})$, where $\gamma_{ij}$ and
$\gamma_{i'j'}$ are consecutive numbers in the sequence of
$\gamma$'s. We verify first that the $\bP^2$-s which are in the
exceptional locus of $X_t$ have lines whose classes are only of type
$\pm\lambda_{rs}$. Secondly, that pairs $(i,j)$ and $(i',j')$ are among
those $(r,s)$ which occur on $X_t$. The sign of $\pm\lambda_{rs}$ will
depend on the position of $\gamma_{rs}$ with respect to $t$. Hence we
flop the $\bP^2$ with lines of type $-\lambda_{i'j'}$ and proceed to
the next interval.  Note that with this single flop we keep the
(relative) projectivity of the model (over $Y$).  The argument will
stop when $X_t$ contains only one $\bP^2$, with lines in the class
$+\lambda_{1n}$.

We run this algorithm for $n=7$ in the next section.

\subsection{Explicit flops}
The aim of the following diagrams is to describe explicitly the
different SQM models obtained as resolutions of the 4-dimensional
symplectic singularity coming from the action of $\bZ_n\wr \bZ_2$, for
$n = 7$ (or less). We use the set up and notation from Theorem
\ref{cone-division} and its proof.  In particular each SQM model
corresponds to a connected component in
$\bR_{>0}\setminus\{\gamma_{ij}\}$, where $\gamma_{ij}$ are the
threshold values associated with flops in \ref{gammas}. The diagram
titled $(\gamma_{ij},\gamma_{i'j'})$ corresponds to the SQM model in
the interval bounded by these thresholds. Each diagram describes the
components of the special fiber of the resolution; these are the only
things which vary from one model to the other.  We actually omit the
components which cannot become $\bP^2$, i.e. the ones which will be
never flopped and which are isomorphic to $F_4$ (see Lemma \ref{qi}).

In each diagram a node corresponds to a component of the special
fiber, i.e.~to a surface.  The isomorphism class of the surface is
denoted by the following codes: $\bured=\bP^2$, $\bugre=F_1$,
$\bubla=\bP^1\times\bP^1$ and $\bublu$ denotes blow-up of $\bP^2$ in
two points (or $\bP^1\times\bP^1$ in one point).

The nodes of two intersecting components are connected by a line
segment: this segment can be dotted, if the incidence between the
components is a point, or solid, if the incidence is a curve. In fact,
one can easily compute the cohomological classes of all incidence
curves in our diagrams; we label only some of them for the sake of
clarity of the picture. We use the classes in cohomology of essential
curves, which, as above, we denote by $e_0, e_1,\dots,e_n$, where
$e_0$ is the class of a fiber over the surface of $\bA_1$
singularities and $e_1,\dots, e_n$ are components of the fiber over
$\bA_n$ singularity. Moreover, for $0\leq i\leq j\leq n$, by
$\lambda_{ij}$ we denote $e_0-(e_i+\cdots+e_j)$.  In the diagrams
below, for instance, $e_i$'s appear as classes of rulings of quadrics
as well as $F_1$'s, while $\pm \lambda_{ij}$'s are classes of lines in
$\bP^2$ or sections of $F_1$'s. For example, incidence curves for
$\bP^1\times\bP^1$, represented as edges of our diagrams at the vertex
denoted by $\bubla$, have the same classes at the opposite ends of the
vertex: e.g.~the class of $\begin{xy}<12pt,0pt>: (-0.7,0)*={} ;
  (0,0)*={\bubla} **@{-}\end{xy}\; $ is the same as of
$\; \begin{xy}<12pt,0pt>: (0.7,0)*={} ; (0,0)*={\bubla}
  **@{-}\end{xy}$. Also, if the class of $\begin{xy}<12pt,0pt>:
  (-0.7,0)*={} ; (0,0)*={\bugre} **@{-}\end{xy}\; $ is $\lambda_{ij}$
and the class of $\; \begin{xy}<12pt,0pt>: (0.7,0)*={} ;
  (0,0)*={\bugre} **@{-}\end{xy}$ is $\lambda_{i+1,j}$ then the class
of ruling, e.g.~the class $\; \begin{xy}<12pt,0pt>: (0,0.5)*={} ;
  (0,0)*={\bugre} **@{-}\end{xy}\; $, is equal to
$\lambda_{ij}-\lambda_{i+1,j}=e_{i+1}$.

We think that the reader has now the necessary instructions to use the
diagrams in order to take the "journey" through the different SQM
models, following the path in $\Mov(X/Y)$ given, in the previous
section, by the half-line $D_0+\bR_{\geq 0}\cdot(-E_0)$.  We can
accompany the reader in the first steps: the first diagram represents
the special fiber in the Hilb-Chow resolution, as described in
\ref{wreath-product}.  To go from this first model to the second, one
has to cross the wall in $\Mov(X/Y)$ given by the hyperplane
$\lambda_{11}^\perp$, that is one has to flop the $\bP^2$ in the lower
node, let us call it $(1,1)$. This will change the sign of the line of
this $\bP^2$ in the cohomology of $X$, from $-\lambda_{11}$ to
$\lambda_{11}$. It will also change the component in the node $(1,2)$,
from $\bublu$, the blow-up of $\bP^1\times\bP^1$ in one point, into
$\bugre$, which is $F_1$. The incidence between the components in the
nodes $(1,1)$ and $(1,2)$ will change from a curve into a point.
The step from the second to the third
model consists in crossing the wall in $\Mov(X/Y)$ given by the hyperplane
$\lambda_{22}^\perp$. This is the flop of the $\bP^2$ in the node $(2,2)$.  It will
change the components in the node $(1,2)$, $(1,3)$ and $(3,2)$ and
their incidence, as indicated in the third diagram. 

The journey will
proceed in this way and it will end in the model described by the last
diagram: here we have only one $\bP^2$, which is in the node $(1,6)$,
a bunch of $F_1$, in the nodes $(1,i)$ and $(j,6)$, $i\not= 6$ and
$j\not= 1$ , and $\bP^1\times\bP^1$ in the other nodes.

$$\begin{array}{ccc}
(0,\gamma_{11})&&(\gamma_{11},\gamma_{22})\\ \\
\begin{xy}<24pt,0pt>:
(0,5)*={\bubla}="F16", (1,5)*={\bubla}="F26", (2,5)*={\bubla}="F36", (3,5)*={\bubla}="F46", 
                                                                     (4,5)*={\bublu}="F56", (5,5)*={\bured}="F66", 
(0,4)*={\bubla}="F15", (1,4)*={\bubla}="F25", (2,4)*={\bubla}="F35", (3,4)*={\bublu}="F45", (4,4)*={\bured}="F55", 
(0,3)*={\bubla}="F14", (1,3)*={\bubla}="F24", (2,3)*={\bublu}="F34", (3,3)*={\bured}="F44",  
(0,2)*={\bubla}="F13", (1,2)*={\bublu}="F23", (2,2)*={\bured}="F33", 
(0,1)*={\bublu}="F12", (1,1)*={\bured}="F22",
(0,0)*={\bured}="F11",
"F11";(0,5.5) **@{-}, "F22";(1,5.5) **@{-}, "F33";(2,5.5) **@{-}, "F44";(3,5.5) **@{-}, "F55";(4,5.5) **@{-},
"F22";(-0.5,1) **@{-}, "F33";(-0.5,2) **@{-}, "F44";(-0.5,3) **@{-}, "F55";(-0.5,4) **@{-}, "F66";(-0.5,5) **@{-},
"F12";(0.5,0.5) **@{-}, "F23";(1.5,1.5) **@{-}, "F34";(2.5,2.5) **@{-}, "F45";(3.5,3.5) **@{-}, "F56";(4.5,4.5) **@{-},
"F22";"F13" **@{.}, "F23";"F14" **@{.}, "F33";"F15" **@{.}, "F34";"F16" **@{.}, "F44";"F26" **@{.}, "F45";"F36" **@{.},
"F55";"F46" **@{.}, "F12";"F56" **@{.}, "F13";"F46" **@{.}, "F14";"F36" **@{.}, "F15";"F26" **@{.},
(0.3,-0.2)*={-\lambda_{11}}, (0.7,0.5)*={e_0}, (1.3,0.8)*={-\lambda_{22}}, (1.7,1.5)*={e_0}, 
(2.3,1.8)*={-\lambda_{33}}, (2.7,2.5)*={e_0}, (3.3,2.8)*={-\lambda_{44}}, (3.7,3.5)*={e_0}, 
(4.3,3.8)*={-\lambda_{55}}, (4.7,4.5)*={e_0}, (5.3,4.8)*={-\lambda_{66}},
(-0.5,1)*={e_2}, (-0.5,2)*={e_3}, (-0.5,3)*={e_4}, (-0.5,4)*={e_5}, (-0.5,5)*={e_6},
(0,5.5)*={e_1}, (1,5.5)*={e_2}, (2,5.5)*={e_3}, (3,5.5)*={e_4}, (4,5.5)*={e_5}
\end{xy}&\phantom{xx}&
\begin{xy}<24pt,0pt>:
(0,5)*={\bubla}="F16", (1,5)*={\bubla}="F26", (2,5)*={\bubla}="F36", (3,5)*={\bubla}="F46", 
                                                                     (4,5)*={\bublu}="F56", (5,5)*={\bured}="F66", 
(0,4)*={\bubla}="F15", (1,4)*={\bubla}="F25", (2,4)*={\bubla}="F35", (3,4)*={\bublu}="F45", (4,4)*={\bured}="F55", 
(0,3)*={\bubla}="F14", (1,3)*={\bubla}="F24", (2,3)*={\bublu}="F34", (3,3)*={\bured}="F44",  
(0,2)*={\bubla}="F13", (1,2)*={\bublu}="F23", (2,2)*={\bured}="F33", 
(0,1)*={\bugre}="F12", (1,1)*={\bured}="F22",
(0,0)*={\bured}="F11",
"F12";(0,5.5) **@{-}, "F22";(1,5.5) **@{-}, "F33";(2,5.5) **@{-}, "F44";(3,5.5) **@{-}, "F55";(4,5.5) **@{-},
"F22";(-0.5,1) **@{-}, "F33";(-0.5,2) **@{-}, "F44";(-0.5,3) **@{-}, "F55";(-0.5,4) **@{-}, "F66";(-0.5,5) **@{-},
"F12";(0.5,0.5) **@{-}, "F23";(1.5,1.5) **@{-}, "F34";(2.5,2.5) **@{-}, "F45";(3.5,3.5) **@{-}, "F56";(4.5,4.5) **@{-},
"F22";"F13" **@{.}, "F23";"F14" **@{.}, "F33";"F15" **@{.}, "F34";"F16" **@{.}, "F44";"F26" **@{.}, "F45";"F36" **@{.},
"F55";"F46" **@{.}, "F12";"F56" **@{.}, "F13";"F46" **@{.}, "F14";"F36" **@{.}, "F15";"F26" **@{.},
"F11";"F12" **@{.},
(0.3,-0.2)*={\lambda_{11}}, (0.7,0.5)*={e_1}, (1.3,0.8)*={-\lambda_{22}}, (1.7,1.5)*={e_0}, 
(2.3,1.8)*={-\lambda_{33}}, (2.7,2.5)*={e_0}, (3.3,2.8)*={-\lambda_{44}}, (3.7,3.5)*={e_0}, 
(4.3,3.8)*={-\lambda_{55}}, (4.7,4.5)*={e_0}, (5.3,4.8)*={-\lambda_{66}},
(-0.5,1)*={-\lambda_{12}}, (-0.5,2)*={e_3}, (-0.5,3)*={e_4}, (-0.5,4)*={e_5}, (-0.5,5)*={e_6},
(0,5.5)*={e_1}, (1,5.5)*={e_2}, (2,5.5)*={e_3}, (3,5.5)*={e_4}, (4,5.5)*={e_5}
\end{xy}
\end{array}$$
$$\begin{array}{ccc}
(\gamma_{22},\gamma_{12})&\phantom{xx}&(\gamma_{12},\gamma_{33})\\ \\
\begin{xy}<24pt,0pt>:
(0,5)*={\bubla}="F16", (1,5)*={\bubla}="F26", (2,5)*={\bubla}="F36", (3,5)*={\bubla}="F46", 
                                                                     (4,5)*={\bublu}="F56", (5,5)*={\bured}="F66", 
(0,4)*={\bubla}="F15", (1,4)*={\bubla}="F25", (2,4)*={\bubla}="F35", (3,4)*={\bublu}="F45", (4,4)*={\bured}="F55", 
(0,3)*={\bubla}="F14", (1,3)*={\bubla}="F24", (2,3)*={\bublu}="F34", (3,3)*={\bured}="F44",  
(0,2)*={\bublu}="F13", (1,2)*={\bugre}="F23", (2,2)*={\bured}="F33", 
(0,1)*={\bured}="F12", (1,1)*={\bured}="F22",
(0,0)*={\bured}="F11",
"F12";(0,5.5) **@{-}, "F33";(2,5.5) **@{-}, "F44";(3,5.5) **@{-}, "F55";(4,5.5) **@{-},
"F23";(1,5.5) **@{-}, "F33";(-0.5,2) **@{-}, "F44";(-0.5,3) **@{-}, "F55";(-0.5,4) **@{-}, "F66";(-0.5,5) **@{-},
"F23";(1.5,1.5) **@{-}, "F34";(2.5,2.5) **@{-}, "F45";(3.5,3.5) **@{-}, "F56";(4.5,4.5) **@{-},
  "F22";"F13" **@{-}, 
"F22";"F13" **@{.}, "F23";"F14" **@{.}, "F33";"F15" **@{.}, "F34";"F16" **@{.}, "F44";"F26" **@{.}, "F45";"F36" **@{.},
"F55";"F46" **@{.}, "F23";"F56" **@{.}, "F13";"F46" **@{.}, "F14";"F36" **@{.}, "F15";"F26" **@{.},
  "F11";"F12" **@{.}, "F12";"F22" **@{.}, "F22";"F23" **@{.},
(0.3,-0.2)*={\lambda_{11}}, (1.3,0.8)*={\lambda_{22}}, (1.7,1.5)*={e_2}, 
(2.3,1.8)*={-\lambda_{33}}, (2.7,2.5)*={e_0}, (3.3,2.8)*={-\lambda_{44}}, (3.7,3.5)*={e_0}, 
(4.3,3.8)*={-\lambda_{55}}, (4.7,4.5)*={e_0}, (5.3,4.8)*={-\lambda_{66}},
(-0.5,1)*={-\lambda_{12}}, (-0.5,2)*={e_3}, (-0.5,3)*={e_4}, (-0.5,4)*={e_5}, (-0.5,5)*={e_6},
(0,5.5)*={e_1}, (1,5.5)*={e_2}, (2,5.5)*={e_3}, (3,5.5)*={e_4}, (4,5.5)*={e_5},
(0.4,2.2)*={-\lambda_{23}}
\end{xy}&&
\begin{xy}<24pt,0pt>:
(0,5)*={\bubla}="F16", (1,5)*={\bubla}="F26", (2,5)*={\bubla}="F36", (3,5)*={\bubla}="F46", 
                                                                     (4,5)*={\bublu}="F56", (5,5)*={\bured}="F66", 
(0,4)*={\bubla}="F15", (1,4)*={\bubla}="F25", (2,4)*={\bubla}="F35", (3,4)*={\bublu}="F45", (4,4)*={\bured}="F55", 
(0,3)*={\bubla}="F14", (1,3)*={\bubla}="F24", (2,3)*={\bublu}="F34", (3,3)*={\bured}="F44",  
(0,2)*={\bugre}="F13", (1,2)*={\bugre}="F23", (2,2)*={\bured}="F33", 
(0,1)*={\bured}="F12", (1,1)*={\bugre}="F22",
(0,0)*={\bugre}="F11",
"F11";(0.5,0) **@{-}, (0.7,0)*={e_2}, "F22";(1,0.6) **@{-}, (1.2,0.6)*={e_1},
"F13";(0,5.5) **@{-}, "F33";(2,5.5) **@{-}, "F44";(3,5.5) **@{-}, "F55";(4,5.5) **@{-},
"F23";(1,5.5) **@{-}, "F33";(-0.5,2) **@{-}, "F44";(-0.5,3) **@{-}, "F55";(-0.5,4) **@{-}, "F66";(-0.5,5) **@{-},
"F23";(1.5,1.5) **@{-}, "F34";(2.5,2.5) **@{-}, "F45";(3.5,3.5) **@{-}, "F56";(4.5,4.5) **@{-},
  "F13";"F22" **@{-}, "F12";(0,-0.5) **@{-}, "F12";(1.5,1) **@{-},
"F22";"F13" **@{.}, "F23";"F14" **@{.}, "F33";"F15" **@{.}, "F34";"F16" **@{.}, "F44";"F26" **@{.}, "F45";"F36" **@{.},
"F55";"F46" **@{.}, "F23";"F56" **@{.}, "F13";"F46" **@{.}, "F14";"F36" **@{.}, "F15";"F26" **@{.},
   "F22";"F23" **@{.},  "F12";"F13" **@{.},  "F11";"F22" **@{.}, 
(-0.2,-0.3)*={\lambda_{11}}, (1.6,1)*={\lambda_{22}}, (1.7,1.5)*={e_2}, 
(2.3,1.8)*={-\lambda_{33}}, (2.7,2.5)*={e_0}, (3.3,2.8)*={-\lambda_{44}}, (3.7,3.5)*={e_0}, 
(4.3,3.8)*={-\lambda_{55}}, (4.7,4.5)*={e_0}, (5.3,4.8)*={-\lambda_{66}},
(-0.2,0.8)*={\lambda_{12}}, (-0.5,2)*={-\lambda_{13}}, (-0.5,3)*={e_4}, (-0.5,4)*={e_5}, (-0.5,5)*={e_6},
(0,5.5)*={e_1}, (1,5.5)*={e_2}, (2,5.5)*={e_3}, (3,5.5)*={e_4}, (4,5.5)*={e_5},
   (0.4,2.2)*={-\lambda_{23}}, (0.6,1.6)*={e_1}
\end{xy}
\end{array}$$

$$\begin{array}{ccc}
(\gamma_{33},\gamma_{23})&\phantom{xx}&(\gamma_{23},\gamma_{13})\\ \\
\begin{xy}<24pt,0pt>:
(0,5)*={\bubla}="F16", (1,5)*={\bubla}="F26", (2,5)*={\bubla}="F36", (3,5)*={\bubla}="F46", 
                                                                     (4,5)*={\bublu}="F56", (5,5)*={\bured}="F66", 
(0,4)*={\bubla}="F15", (1,4)*={\bubla}="F25", (2,4)*={\bubla}="F35", (3,4)*={\bublu}="F45", (4,4)*={\bured}="F55", 
(0,3)*={\bubla}="F14", (1,3)*={\bublu}="F24", (2,3)*={\bugre}="F34", (3,3)*={\bured}="F44",  
(0,2)*={\bugre}="F13", (1,2)*={\bured}="F23", (2,2)*={\bured}="F33", 
(0,1)*={\bured}="F12", (1,1)*={\bugre}="F22",
(0,0)*={\bugre}="F11",
"F11";(0.5,0) **@{-}, (0.7,0)*={e_2}, "F22";(1,0.6) **@{-}, (1.2,0.6)*={e_1},
"F13";(0,5.5) **@{-}, "F34";(2,5.5) **@{-}, "F44";(3,5.5) **@{-}, "F55";(4,5.5) **@{-},
"F23";(1,5.5) **@{-}, "F23";(-0.5,2) **@{-}, "F44";(-0.5,3) **@{-}, "F55";(-0.5,4) **@{-}, "F66";(-0.5,5) **@{-},
"F34";(2.5,2.5) **@{-}, "F45";(3.5,3.5) **@{-}, "F56";(4.5,4.5) **@{-},
  "F13";"F22" **@{-}, "F12";(0,-0.5) **@{-}, "F12";(1.5,1) **@{-}, "F33";"F24" **@{-},
"F22";"F13" **@{.}, "F23";"F14" **@{.}, "F24";"F15" **@{.}, "F34";"F16" **@{.}, "F44";"F26" **@{.}, "F45";"F36" **@{.},
"F55";"F46" **@{.}, "F34";"F56" **@{.}, "F13";"F46" **@{.}, "F14";"F36" **@{.}, "F15";"F26" **@{.},
   "F22";"F23" **@{.},  "F12";"F13" **@{.},  "F11";"F22" **@{.}, "F23";"F33" **@{.}, "F34";"F33" **@{.},
(-0.2,-0.3)*={\lambda_{11}}, (1.6,1)*={\lambda_{22}}, (1.2,1.8)*={-\lambda_{23}}, 
(2.3,1.8)*={\lambda_{33}}, (2.7,2.5)*={e_3}, 
(3.3,2.8)*={-\lambda_{44}}, (3.7,3.5)*={e_0}, 
(4.3,3.8)*={-\lambda_{55}}, (4.7,4.5)*={e_0}, (5.3,4.8)*={-\lambda_{66}},
(-0.2,0.8)*={\lambda_{12}}, (-0.4,1.8)*={-\lambda_{13}}, (-0.5,3)*={e_4}, (-0.5,4)*={e_5}, (-0.5,5)*={e_6},
(0,5.5)*={e_1}, (1,5.5)*={e_2}, (2,5.5)*={e_3}, (3,5.5)*={e_4}, (4,5.5)*={e_5},
   (0.6,1.6)*={e_1}, (1.5,2.9)*={-\lambda_{34}},
\end{xy}&&
\begin{xy}<24pt,0pt>:
(0,5)*={\bubla}="F16", (1,5)*={\bubla}="F26", (2,5)*={\bubla}="F36", (3,5)*={\bubla}="F46", 
                                                                     (4,5)*={\bublu}="F56", (5,5)*={\bured}="F66", 
(0,4)*={\bubla}="F15", (1,4)*={\bubla}="F25", (2,4)*={\bubla}="F35", (3,4)*={\bublu}="F45", (4,4)*={\bured}="F55", 
(0,3)*={\bublu}="F14", (1,3)*={\bugre}="F24", (2,3)*={\bugre}="F34", (3,3)*={\bured}="F44",  
(0,2)*={\bured}="F13", (1,2)*={\bured}="F23", (2,2)*={\bugre}="F33", 
(0,1)*={\bured}="F12", (1,1)*={\bublu}="F22",
(0,0)*={\bugre}="F11",
"F11";(0.5,0) **@{-}, (0.7,0)*={e_2}, "F23";(1,0.6) **@{-}, (1.2,0.6)*={e_1}, "F33";(2,1.6) **@{-}, (2.2,1.6)*={e_2}, 
"F13";(0,5.5) **@{-}, "F34";(2,5.5) **@{-}, "F44";(3,5.5) **@{-}, "F55";(4,5.5) **@{-},
"F24";(1,5.5) **@{-}, "F44";(-0.5,3) **@{-}, "F55";(-0.5,4) **@{-}, "F66";(-0.5,5) **@{-},
"F34";(2.5,2.5) **@{-}, "F45";(3.5,3.5) **@{-}, "F56";(4.5,4.5) **@{-},
"F13";"F22" **@{-}, "F12";(0,-0.5) **@{-}, "F12";(1.5,1) **@{-}, "F33";"F24" **@{-},
"F23";(2.4,2) **@{-}, "F23";"F14" **@{-},
"F22";"F13" **@{.}, "F24";"F15" **@{.}, "F34";"F16" **@{.}, "F44";"F26" **@{.}, "F45";"F36" **@{.},
"F55";"F46" **@{.}, "F34";"F56" **@{.}, "F24";"F46" **@{.}, "F14";"F36" **@{.}, "F15";"F26" **@{.},
    "F12";"F13" **@{.},  "F11";"F33" **@{.},  "F34";"F33" **@{.}, "F13";"F23" **@{.}, "F23";"F24" **@{.},
(-0.2,-0.3)*={\lambda_{11}}, (1.6,1)*={e_3}, (1.1,1.8)*={\lambda_{23}}, 
(1.7,2.5)*={e_2}, (2.7,2.5)*={e_3}, (3.3,2.8)*={-\lambda_{44}}, (3.7,3.5)*={e_0}, 
(4.3,3.8)*={-\lambda_{55}}, (4.7,4.5)*={e_0}, (5.3,4.8)*={-\lambda_{66}}, (2.5,2)*={\lambda_{33}},
(-0.2,0.8)*={\lambda_{12}}, (-0.4,1.8)*={-\lambda_{13}}, (-0.5,3)*={e_4}, (-0.5,4)*={e_5}, (-0.5,5)*={e_6},
(0,5.5)*={e_1}, (1,5.5)*={e_2}, (2,5.5)*={e_3}, (3,5.5)*={e_4}, (4,5.5)*={e_5}, 
(1.5,2.9)*={-\lambda_{34}}, (0.5,2.9)*={-\lambda_{24}}, 
\end{xy}
\end{array}$$
$$\begin{array}{ccc}
(\gamma_{13},\gamma_{44})&\phantom{xx}&(\gamma_{44},\gamma_{34})\\ \\
\begin{xy}<24pt,0pt>:
(0,5)*={\bubla}="F16", (1,5)*={\bubla}="F26", (2,5)*={\bubla}="F36", (3,5)*={\bubla}="F46", 
                                                                     (4,5)*={\bublu}="F56", (5,5)*={\bured}="F66", 
(0,4)*={\bubla}="F15", (1,4)*={\bubla}="F25", (2,4)*={\bubla}="F35", (3,4)*={\bublu}="F45", (4,4)*={\bured}="F55", 
(0,3)*={\bugre}="F14", (1,3)*={\bugre}="F24", (2,3)*={\bugre}="F34", (3,3)*={\bured}="F44",  
(0,2)*={\bured}="F13", (1,2)*={\bugre}="F23", (2,2)*={\bugre}="F33", 
(0,1)*={\bugre}="F12", (1,1)*={\bubla}="F22",
(0,0)*={\bugre}="F11",
"F11";(0.5,0) **@{-}, (0.7,0)*={e_2}, "F23";(1,0.6) **@{-}, (1.2,0.6)*={e_1}, "F33";(2,1.6) **@{-}, (2.2,1.6)*={e_2}, 
"F14";(0,5.5) **@{-}, "F34";(2,5.5) **@{-}, "F44";(3,5.5) **@{-}, "F55";(4,5.5) **@{-},
"F24";(1,5.5) **@{-}, "F44";(-0.5,3) **@{-}, "F55";(-0.5,4) **@{-}, "F66";(-0.5,5) **@{-},
"F34";(2.5,2.5) **@{-}, "F45";(3.5,3.5) **@{-}, "F56";(4.5,4.5) **@{-},
 "F12";(0,-0.5) **@{-}, "F12";(1.5,1) **@{-}, "F33";"F24" **@{-},
  "F23";(2.4,2) **@{-},  "F23";"F14" **@{-},  "F13";"F23" **@{-}, "F12";"F13" **@{-}, 
"F22";"F13" **@{.}, "F24";"F15" **@{.}, "F34";"F16" **@{.}, "F44";"F26" **@{.}, "F45";"F36" **@{.},
"F55";"F46" **@{.}, "F34";"F56" **@{.}, "F24";"F46" **@{.}, "F14";"F36" **@{.}, "F15";"F26" **@{.},
"F11";"F33" **@{.},  "F34";"F33" **@{.}, "F23";"F24" **@{.},  "F13";"F22" **@{.}, "F13";"F14" **@{.},
"F12";"F23" **@{.},
(-0.2,-0.3)*={\lambda_{11}}, (1.6,1)*={e_3}, (1.6,1.9)*={\lambda_{23}}, 
(1.7,2.5)*={e_2}, (2.7,2.5)*={e_3}, (3.3,2.8)*={-\lambda_{44}}, (3.7,3.5)*={e_0}, 
(4.3,3.8)*={-\lambda_{55}}, (4.7,4.5)*={e_0}, (5.3,4.8)*={-\lambda_{66}}, (2.5,2)*={\lambda_{33}},
(-0.2,0.6)*={\lambda_{12}}, (-0.3,1.8)*={\lambda_{13}}, (-0.4,2.8)*={-\lambda_{14}}, 
(-0.5,4)*={e_5}, (-0.5,5)*={e_6},
(0,5.5)*={e_1}, (1,5.5)*={e_2}, (2,5.5)*={e_3}, (3,5.5)*={e_4}, (4,5.5)*={e_5}, (1.5,2.9)*={-\lambda_{34}},
(0.5,2.9)*={-\lambda_{24}}, (1.2,1.5)*={e_1},
\end{xy}&\phantom{xx}&
\begin{xy}<24pt,0pt>:
(0,5)*={\bubla}="F16", (1,5)*={\bubla}="F26", (2,5)*={\bubla}="F36", (3,5)*={\bubla}="F46", 
                                                                     (4,5)*={\bublu}="F56", (5,5)*={\bured}="F66", 
(0,4)*={\bubla}="F15", (1,4)*={\bubla}="F25", (2,4)*={\bublu}="F35", (3,4)*={\bugre}="F45", (4,4)*={\bured}="F55", 
(0,3)*={\bugre}="F14", (1,3)*={\bugre}="F24", (2,3)*={\bured}="F34", (3,3)*={\bured}="F44",  
(0,2)*={\bured}="F13", (1,2)*={\bugre}="F23", (2,2)*={\bugre}="F33", 
(0,1)*={\bugre}="F12", (1,1)*={\bubla}="F22",
(0,0)*={\bugre}="F11",
"F11";(0.5,0) **@{-}, (0.7,0)*={e_2},"F23";(1,0.6) **@{-}, (1.2,0.6)*={e_1}, "F33";(2,1.6) **@{-}, (2.2,1.6)*={e_2}, 
"F14";(0,5.5) **@{-}, "F34";(2,5.5) **@{-}, "F45";(3,5.5) **@{-}, "F55";(4,5.5) **@{-},
"F24";(1,5.5) **@{-}, "F34";(-0.5,3) **@{-}, "F55";(-0.5,4) **@{-}, "F66";(-0.5,5) **@{-},
"F45";(3.5,3.5) **@{-}, "F56";(4.5,4.5) **@{-},
 "F12";(0,-0.5) **@{-}, "F12";(1.5,1) **@{-}, "F33";"F24" **@{-},
 "F23";(2.4,2) **@{-},  "F23";"F14" **@{-},  "F13";"F23" **@{-}, "F12";"F13" **@{-}, 
 "F44";"F35" **@{-}, 
"F22";"F13" **@{.}, "F24";"F15" **@{.}, "F34";"F16" **@{.}, "F35";"F26" **@{.}, "F45";"F36" **@{.},
"F55";"F46" **@{.}, "F45";"F56" **@{.}, "F24";"F46" **@{.}, "F14";"F36" **@{.}, "F15";"F26" **@{.},
"F11";"F33" **@{.}, "F34";"F33" **@{.}, "F23";"F24" **@{.},  "F13";"F22" **@{.}, "F13";"F14" **@{.},
"F34";"F44" **@{.}, "F44";"F45" **@{.},
"F12";"F23" **@{.},
(-0.2,-0.3)*={\lambda_{11}}, (1.6,1)*={e_3}, (1.6,1.9)*={\lambda_{23}}, 
(1.7,2.5)*={e_2}, (3.3,2.8)*={\lambda_{44}}, (3.7,3.5)*={e_4}, 
(4.3,3.8)*={-\lambda_{55}}, (4.7,4.5)*={e_0}, (5.3,4.8)*={-\lambda_{66}}, (2.5,2)*={\lambda_{33}},
(-0.2,0.6)*={\lambda_{12}}, (-0.3,1.8)*={\lambda_{13}}, (-0.5,4)*={e_5}, (-0.5,5)*={e_6},
(0,5.5)*={e_1}, (1,5.5)*={e_2}, (2,5.5)*={e_3}, (3,5.5)*={e_4}, (4,5.5)*={e_5}, (1.8,2.9)*={-\lambda_{34}},
(0.5,2.9)*={-\lambda_{24}}, (1.2,1.5)*={e_1}, (2.5,3.9)*={-\lambda_{45}},  (-0.4,2.8)*={-\lambda_{14}}, 
\end{xy}\end{array}
$$

$$\begin{array}{ccc}
(\gamma_{34},\gamma_{24})&\phantom{xx}&(\gamma_{24},\gamma_{14})\\ \\
\begin{xy}<24pt,0pt>:
(0,5)*={\bubla}="F16", (1,5)*={\bubla}="F26", (2,5)*={\bubla}="F36", (3,5)*={\bubla}="F46", 
                                                                     (4,5)*={\bublu}="F56", (5,5)*={\bured}="F66", 
(0,4)*={\bubla}="F15", (1,4)*={\bublu}="F25", (2,4)*={\bugre}="F35", (3,4)*={\bugre}="F45", (4,4)*={\bured}="F55", 
(0,3)*={\bugre}="F14", (1,3)*={\bured}="F24", (2,3)*={\bured}="F34", (3,3)*={\bugre}="F44",  
(0,2)*={\bured}="F13", (1,2)*={\bugre}="F23", (2,2)*={\bublu}="F33", 
(0,1)*={\bugre}="F12", (1,1)*={\bubla}="F22",
(0,0)*={\bugre}="F11",
"F11";(0.5,0) **@{-}, (0.7,0)*={e_2},
"F23";(1,0.6) **@{-}, (1.2,0.6)*={e_1}, "F34";(2,1.6) **@{-}, (2.2,1.6)*={e_2}, "F44";(3,2.6) **@{-}, (3.2,2.6)*={e_3}, 
"F14";(0,5.5) **@{-}, "F35";(2,5.5) **@{-}, "F45";(3,5.5) **@{-}, "F55";(4,5.5) **@{-},
"F24";(1,5.5) **@{-}, "F24";(-0.5,3) **@{-}, "F55";(-0.5,4) **@{-}, "F66";(-0.5,5) **@{-},
"F45";(3.5,3.5) **@{-}, "F56";(4.5,4.5) **@{-},
"F12";(0,-0.5) **@{-}, "F12";(1.5,1) **@{-}, "F33";"F24" **@{-},
 "F23";(2.5,2) **@{-}, "F23";"F14" **@{-},  "F13";"F23" **@{-}, "F12";"F13" **@{-}, 
 "F44";"F35" **@{-}, "F25";"F34" **@{-}, "F34";(3.4,3) **@{-},
"F22";"F13" **@{.}, "F24";"F15" **@{.}, "F25";"F16" **@{.}, "F35";"F26" **@{.}, "F45";"F36" **@{.},
"F55";"F46" **@{.}, "F45";"F56" **@{.}, "F35";"F46" **@{.}, "F14";"F36" **@{.}, "F15";"F26" **@{.},
"F11";"F44" **@{.}, "F23";"F24" **@{.}, "F13";"F22" **@{.}, "F13";"F14" **@{.}, 
"F44";"F45" **@{.}, "F24";"F34" **@{.}, "F34";"F35" **@{.},
"F12";"F23" **@{.},
(-0.2,-0.3)*={\lambda_{11}}, (1.6,1)*={e_3}, (2.6,2)*={e_4}, (1.6,1.9)*={\lambda_{23}}, 
(3.5,3)*={\lambda_{44}}, 
(3.7,3.5)*={e_4}, 
(4.3,3.8)*={-\lambda_{55}}, (4.7,4.5)*={e_0}, (5.3,4.8)*={-\lambda_{66}},
(-0.2,0.6)*={\lambda_{12}}, (-0.3,1.8)*={\lambda_{13}}, (-0.4,2.9)*={-\lambda_{14}}, (-0.5,4)*={e_5}, (-0.5,5)*={e_6},
(0,5.5)*={e_1}, (1,5.5)*={e_2}, (2,5.5)*={e_3}, (3,5.5)*={e_4}, (4,5.5)*={e_5}, (2.3,2.9)*={\lambda_{34}},
(1.3,2.8)*={-\lambda_{24}}, (1.2,1.5)*={e_1}, 
(2.5,3.9)*={-\lambda_{45}},(1.5,3.9)*={-\lambda_{35}}, 

\end{xy}&&
\begin{xy}<24pt,0pt>:
(0,5)*={\bubla}="F16", (1,5)*={\bubla}="F26", (2,5)*={\bubla}="F36", (3,5)*={\bubla}="F46", 
                                                                     (4,5)*={\bublu}="F56", (5,5)*={\bured}="F66", 
(0,4)*={\bublu}="F15", (1,4)*={\bugre}="F25", (2,4)*={\bugre}="F35", (3,4)*={\bugre}="F45", (4,4)*={\bured}="F55", 
(0,3)*={\bured}="F14", (1,3)*={\bured}="F24", (2,3)*={\bugre}="F34", (3,3)*={\bugre}="F44",  
(0,2)*={\bured}="F13", (1,2)*={\bublu}="F23", (2,2)*={\bubla}="F33", 
(0,1)*={\bugre}="F12", (1,1)*={\bubla}="F22",
(0,0)*={\bugre}="F11",
"F11";(0.5,0) **@{-}, (0.7,0)*={e_2},
"F24";(1,0.6) **@{-}, (1.2,0.6)*={e_1}, "F34";(2,1.6) **@{-}, (2.2,1.6)*={e_2}, "F44";(3,2.6) **@{-}, (3.2,2.6)*={e_3}, 
"F14";(0,5.5) **@{-}, "F35";(2,5.5) **@{-}, "F45";(3,5.5) **@{-}, "F55";(4,5.5) **@{-},
"F25";(1,5.5) **@{-},  "F55";(-0.5,4) **@{-}, "F66";(-0.5,5) **@{-},
"F45";(3.5,3.5) **@{-}, "F56";(4.5,4.5) **@{-},
"F12";(0,-0.5) **@{-}, "F12";(1.5,1) **@{-}, 
"F23";(2.5,2) **@{-}, "F23";"F14" **@{-},  "F13";"F23" **@{-}, "F12";"F13" **@{-}, 
"F44";"F35" **@{-}, "F25";"F34" **@{-}, "F34";(3.4,3) **@{-},  
"F24";"F15" **@{-},  "F24";"F34" **@{-}, 
"F22";"F13" **@{.}, "F25";"F16" **@{.}, "F35";"F26" **@{.}, "F45";"F36" **@{.},
"F55";"F46" **@{.}, "F45";"F56" **@{.}, "F35";"F46" **@{.}, "F25";"F36" **@{.}, "F15";"F26" **@{.},
"F11";"F44" **@{.}, "F13";"F22" **@{.}, "F13";"F14" **@{.}, 
"F44";"F45" **@{.}, "F34";"F35" **@{.}, "F33";"F24" **@{.},"F24";"F14" **@{.}, "F24";"F25" **@{.},
"F12";"F34" **@{.},
(-0.2,-0.3)*={\lambda_{11}}, (1.6,1)*={e_3}, (2.6,2)*={e_4},
(3.5,3)*={\lambda_{44}}, 
(3.7,3.5)*={e_4}, 
(4.3,3.8)*={-\lambda_{55}}, (4.7,4.5)*={e_0}, (5.3,4.8)*={-\lambda_{66}},
(-0.2,0.6)*={\lambda_{12}}, (-0.3,1.8)*={\lambda_{13}}, (-0.4,2.9)*={-\lambda_{14}}, (-0.5,4)*={e_5}, (-0.5,5)*={e_6},
(0,5.5)*={e_1}, (1,5.5)*={e_2}, (2,5.5)*={e_3}, (3,5.5)*={e_4}, (4,5.5)*={e_5}, (2.5,2.9)*={\lambda_{34}},
(1.3,2.8)*={\lambda_{24}}, (1.2,1.5)*={e_1},  (2.2,2.5)*={e_2}, 
(2.5,3.9)*={-\lambda_{45}},(1.5,3.9)*={-\lambda_{35}},(0.5,3.9)*={-\lambda_{25}},
\end{xy}\end{array}$$

$$\begin{array}{ccc}
(\gamma_{14},\gamma_{55})&\phantom{xx}&(\gamma_{55},\gamma_{45})\\ \\
\begin{xy}<24pt,0pt>:
(0,5)*={\bubla}="F16", (1,5)*={\bubla}="F26", (2,5)*={\bubla}="F36", (3,5)*={\bubla}="F46", 
                                                                     (4,5)*={\bublu}="F56", (5,5)*={\bured}="F66", 
(0,4)*={\bugre}="F15", (1,4)*={\bugre}="F25", (2,4)*={\bugre}="F35", (3,4)*={\bugre}="F45", (4,4)*={\bured}="F55", 
(0,3)*={\bured}="F14", (1,3)*={\bugre}="F24", (2,3)*={\bugre}="F34", (3,3)*={\bugre}="F44",  
(0,2)*={\bugre}="F13", (1,2)*={\bubla}="F23", (2,2)*={\bubla}="F33", 
(0,1)*={\bugre}="F12", (1,1)*={\bubla}="F22",
(0,0)*={\bugre}="F11",
"F11";(0.5,0) **@{-}, (0.7,0)*={e_2},
"F24";(1,0.6) **@{-}, (1.2,0.6)*={e_1}, "F34";(2,1.6) **@{-}, (2.2,1.6)*={e_2}, "F44";(3,2.6) **@{-}, (3.2,2.6)*={e_3},
"F15";(0,5.5) **@{-}, "F35";(2,5.5) **@{-}, "F45";(3,5.5) **@{-}, "F55";(4,5.5) **@{-},
"F25";(1,5.5) **@{-},  "F55";(-0.5,4) **@{-}, "F66";(-0.5,5) **@{-},
"F45";(3.5,3.5) **@{-}, "F56";(4.5,4.5) **@{-},
"F14";(0,-0.5) **@{-}, "F12";(1.5,1) **@{-}, 
"F23";(2.5,2) **@{-},  "F13";"F23" **@{-},
"F44";"F35" **@{-}, "F25";"F34" **@{-}, "F34";(3.4,3) **@{-},  
"F24";"F15" **@{-},  "F24";"F34" **@{-}, "F24";"F14" **@{-},
"F22";"F13" **@{.}, "F25";"F16" **@{.}, "F35";"F26" **@{.}, "F45";"F36" **@{.},
"F55";"F46" **@{.}, "F45";"F56" **@{.}, "F35";"F46" **@{.}, "F25";"F36" **@{.}, "F15";"F26" **@{.},
"F11";"F44" **@{.}, "F13";"F22" **@{.}, "F13";"F14" **@{.}, 
"F44";"F45" **@{.}, "F34";"F35" **@{.}, "F33";"F24" **@{.}, "F24";"F25" **@{.}, "F23";"F14" **@{.},
"F14";"F15" **@{.}, "F12";"F34" **@{.}, "F13";"F24" **@{.},
(-0.2,-0.3)*={\lambda_{11}}, (1.6,1)*={e_3}, (2.6,2)*={e_4},
(3.5,3)*={\lambda_{44}}, 
(3.7,3.5)*={e_4}, 
(4.3,3.8)*={-\lambda_{55}}, (4.7,4.5)*={e_0}, (5.3,4.8)*={-\lambda_{66}},
(-0.2,0.6)*={\lambda_{12}}, (-0.2,1.6)*={\lambda_{13}}, (-0.4,2.9)*={\lambda_{14}},  (-0.5,5)*={e_6},
(0,5.5)*={e_1}, (1,5.5)*={e_2}, (2,5.5)*={e_3}, (3,5.5)*={e_4}, (4,5.5)*={e_5},
(2.5,2.9)*={\lambda_{34}}, (1.5,2.9)*={\lambda_{24}},
(1.2,1.5)*={e_1},  (2.2,2.5)*={e_2}, 
(2.5,3.9)*={-\lambda_{45}},(1.5,3.9)*={-\lambda_{35}},(0.5,3.9)*={-\lambda_{25}},(-0.5,3.9)*={-\lambda_{15}},
\end{xy}&&
\begin{xy}<24pt,0pt>:
(0,5)*={\bubla}="F16", (1,5)*={\bubla}="F26", (2,5)*={\bubla}="F36", (3,5)*={\bublu}="F46", 
                                                                     (4,5)*={\bugre}="F56", (5,5)*={\bured}="F66", 
(0,4)*={\bugre}="F15", (1,4)*={\bugre}="F25", (2,4)*={\bugre}="F35", (3,4)*={\bured}="F45", (4,4)*={\bured}="F55", 
(0,3)*={\bured}="F14", (1,3)*={\bugre}="F24", (2,3)*={\bugre}="F34", (3,3)*={\bugre}="F44",  
(0,2)*={\bugre}="F13", (1,2)*={\bubla}="F23", (2,2)*={\bubla}="F33", 
(0,1)*={\bugre}="F12", (1,1)*={\bubla}="F22",
(0,0)*={\bugre}="F11",
"F11";(0.5,0) **@{-}, (0.7,0)*={e_2},
"F24";(1,0.6) **@{-}, (1.2,0.6)*={e_1}, "F34";(2,1.6) **@{-}, (2.2,1.6)*={e_2}, "F44";(3,2.6) **@{-}, (3.2,2.6)*={e_3},
"F15";(0,5.5) **@{-}, "F35";(2,5.5) **@{-}, "F45";(3,5.5) **@{-}, "F56";(4,5.5) **@{-},
"F25";(1,5.5) **@{-}, "F45";(-0.5,4) **@{-}, "F66";(-0.5,5) **@{-}, "F56";(4.5,4.5) **@{-},
"F14";(0,-0.5) **@{-}, "F12";(1.5,1) **@{-}, 
"F23";(2.5,2) **@{-},   "F13";"F23" **@{-},
"F44";"F35" **@{-}, "F25";"F34" **@{-}, "F34";(3.4,3) **@{-},  
"F24";"F15" **@{-},  "F24";"F34" **@{-}, "F24";"F14" **@{-},
"F55";"F46" **@{-}, 
"F22";"F13" **@{.}, "F25";"F16" **@{.}, "F35";"F26" **@{.}, "F45";"F36" **@{.},
"F35";"F46" **@{.}, "F25";"F36" **@{.}, "F15";"F26" **@{.},
"F11";"F44" **@{.}, "F13";"F22" **@{.}, "F13";"F14" **@{.}, 
"F44";"F45" **@{.}, "F34";"F35" **@{.}, "F33";"F24" **@{.}, "F24";"F25" **@{.}, "F23";"F14" **@{.},
"F14";"F15" **@{.}, "F12";"F34" **@{.}, "F13";"F24" **@{.},
"F45";"F55" **@{.}, "F55";"F56" **@{.}, 
(-0.2,-0.3)*={\lambda_{11}}, (1.6,1)*={e_3}, (2.6,2)*={e_4},
(3.5,3)*={\lambda_{44}}, 
(4.3,3.8)*={\lambda_{55}}, (4.7,4.5)*={e_5}, (5.3,4.8)*={-\lambda_{66}},
(-0.2,0.6)*={\lambda_{12}}, (-0.2,1.6)*={\lambda_{13}}, (-0.4,2.9)*={\lambda_{14}},  (-0.5,5)*={e_6},
(0,5.5)*={e_1}, (1,5.5)*={e_2}, (2,5.5)*={e_3}, (3,5.5)*={e_4}, (4,5.5)*={e_5},
(2.5,2.9)*={\lambda_{34}}, (1.5,2.9)*={\lambda_{24}},
(1.2,1.5)*={e_1},  (2.2,2.5)*={e_2}, 
(3.1,3.8)*={-\lambda_{45}}, (1.5,3.9)*={-\lambda_{35}},(0.5,3.9)*={-\lambda_{25}},(-0.5,3.9)*={-\lambda_{15}},
(3.5,4.9)*={-\lambda_{56}}, 
\end{xy}
\end{array}$$

$$\begin{array}{ccc}
(\gamma_{45},\gamma_{35})&\phantom{xx}&(\gamma_{35},\gamma_{25})\\ \\
\begin{xy}<24pt,0pt>:
(0,5)*={\bubla}="F16", (1,5)*={\bubla}="F26", (2,5)*={\bublu}="F36", (3,5)*={\bugre}="F46", 
                                                                     (4,5)*={\bugre}="F56", (5,5)*={\bured}="F66", 
(0,4)*={\bugre}="F15", (1,4)*={\bugre}="F25", (2,4)*={\bured}="F35", (3,4)*={\bured}="F45", (4,4)*={\bugre}="F55", 
(0,3)*={\bured}="F14", (1,3)*={\bugre}="F24", (2,3)*={\bugre}="F34", (3,3)*={\bublu}="F44",  
(0,2)*={\bugre}="F13", (1,2)*={\bubla}="F23", (2,2)*={\bubla}="F33", 
(0,1)*={\bugre}="F12", (1,1)*={\bubla}="F22",
(0,0)*={\bugre}="F11",
"F11";(0.5,0) **@{-}, (0.7,0)*={e_2},
"F24";(1,0.6) **@{-}, (1.2,0.6)*={e_1}, "F34";(2,1.6) **@{-}, (2.2,1.6)*={e_2}, "F45";(3,2.6) **@{-}, (3.2,2.6)*={e_3},
"F55";(4,3.6) **@{-}, (4.2,3.6)*={e_4},
"F15";(0,5.5) **@{-}, "F35";(2,5.5) **@{-}, "F46";(3,5.5) **@{-}, "F56";(4,5.5) **@{-},
"F25";(1,5.5) **@{-}, "F35";(-0.5,4) **@{-}, "F66";(-0.5,5) **@{-}, "F56";(4.5,4.5) **@{-},
"F14";(0,-0.5) **@{-}, "F12";(1.5,1) **@{-}, 
"F23";(2.5,2) **@{-}, "F13";"F23" **@{-},
"F44";"F35" **@{-}, "F25";"F34" **@{-},  "F34";(3.4,3) **@{-},  
"F24";"F15" **@{-},  "F24";"F34" **@{-}, "F24";"F14" **@{-},
"F55";"F46" **@{-}, "F45";(4.5,4) **@{-}, "F45";"F36" **@{-},
"F22";"F13" **@{.}, "F25";"F16" **@{.}, "F35";"F26" **@{.},
"F25";"F36" **@{.}, "F15";"F26" **@{.},
"F11";"F55" **@{.}, "F13";"F22" **@{.}, "F13";"F14" **@{.}, 
"F34";"F35" **@{.}, "F33";"F24" **@{.}, "F24";"F25" **@{.}, "F23";"F14" **@{.},
"F14";"F15" **@{.}, "F12";"F34" **@{.}, "F13";"F24" **@{.},
"F55";"F56" **@{.}, "F35";"F45" **@{.}, "F45";"F46" **@{.}, 
(-0.2,-0.3)*={\lambda_{11}}, (1.6,1)*={e_3}, (2.6,2)*={e_4},
(3.5,3)*={e_5}, 
(4.4,4)*={\lambda_{55}}, (4.7,4.5)*={e_5}, (5.3,4.8)*={-\lambda_{66}},
(-0.2,0.6)*={\lambda_{12}}, (-0.2,1.6)*={\lambda_{13}}, (-0.4,2.9)*={\lambda_{14}},  (-0.5,5)*={e_6},
(0,5.5)*={e_1}, (1,5.5)*={e_2}, (2,5.5)*={e_3}, (3,5.5)*={e_4}, (4,5.5)*={e_5},
(2.5,2.9)*={\lambda_{34}}, (1.5,2.9)*={\lambda_{24}}, (1.2,1.5)*={e_1},  (2.2,2.5)*={e_2}, 
(3.3,3.8)*={\lambda_{45}}, (1.5,3.9)*={-\lambda_{35}},(0.5,3.9)*={-\lambda_{25}},(-0.5,3.9)*={-\lambda_{15}},
(3.5,4.9)*={-\lambda_{56}}, (2.5,4.9)*={-\lambda_{46}},
\end{xy}&\phantom{xx}&
\begin{xy}<24pt,0pt>:
(0,5)*={\bubla}="F16", (1,5)*={\bublu}="F26", (2,5)*={\bugre}="F36", (3,5)*={\bugre}="F46", 
                                                                     (4,5)*={\bugre}="F56", (5,5)*={\bured}="F66", 
(0,4)*={\bugre}="F15", (1,4)*={\bured}="F25", (2,4)*={\bured}="F35", (3,4)*={\bugre}="F45", (4,4)*={\bugre}="F55", 
(0,3)*={\bured}="F14", (1,3)*={\bugre}="F24", (2,3)*={\bublu}="F34", (3,3)*={\bubla}="F44",  
(0,2)*={\bugre}="F13", (1,2)*={\bubla}="F23", (2,2)*={\bubla}="F33", 
(0,1)*={\bugre}="F12", (1,1)*={\bubla}="F22",
(0,0)*={\bugre}="F11",
"F11";(0.5,0) **@{-}, (0.7,0)*={e_2},
"F24";(1,0.6) **@{-}, (1.2,0.6)*={e_1}, "F35";(2,1.6) **@{-}, (2.2,1.6)*={e_2}, "F45";(3,2.6) **@{-}, (3.2,2.6)*={e_3},
"F55";(4,3.6) **@{-}, (4.2,3.6)*={e_4},
"F15";(0,5.5) **@{-}, "F36";(2,5.5) **@{-}, "F46";(3,5.5) **@{-}, "F56";(4,5.5) **@{-},
"F25";(1,5.5) **@{-}, "F25";(-0.5,4) **@{-}, "F66";(-0.5,5) **@{-}, "F56";(4.5,4.5) **@{-},
"F14";(0,-0.5) **@{-}, "F12";(1.5,1) **@{-}, 
"F23";(2.5,2) **@{-},   "F13";"F23" **@{-},
"F25";"F34" **@{-}, "F34";(3.4,3) **@{-},  
"F24";"F15" **@{-},  "F24";"F34" **@{-}, "F24";"F14" **@{-},
"F55";"F46" **@{-},  "F35";(4.5,4) **@{-}, "F45";"F36" **@{-}, "F35";"F26" **@{-},
"F22";"F13" **@{.}, "F25";"F16" **@{.},
"F15";"F26" **@{.},
"F11";"F55" **@{.}, "F13";"F22" **@{.}, "F13";"F14" **@{.}, 
"F33";"F24" **@{.}, "F24";"F25" **@{.}, "F23";"F14" **@{.},
"F14";"F15" **@{.}, "F12";"F45" **@{.}, "F13";"F24" **@{.},
"F55";"F56" **@{.}, "F45";"F46" **@{.}, "F44";"F35" **@{.},
"F25";"F35" **@{.}, "F35";"F36" **@{.},
(-0.2,-0.3)*={\lambda_{11}}, (1.6,1)*={e_3}, (2.6,2)*={e_4},
(3.5,3)*={e_5}, 
(4.4,4)*={\lambda_{55}}, (4.7,4.5)*={e_5}, (5.3,4.8)*={-\lambda_{66}},
(-0.2,0.6)*={\lambda_{12}}, (-0.2,1.6)*={\lambda_{13}}, (-0.4,2.9)*={\lambda_{14}},  (-0.5,5)*={e_6},
(0,5.5)*={e_1}, (1,5.5)*={e_2}, (2,5.5)*={e_3}, (3,5.5)*={e_4}, (4,5.5)*={e_5},
(1.5,2.9)*={\lambda_{24}}, (1.2,1.5)*={e_1},  (2.2,2.5)*={e_2}, (3.5,3.9)*={\lambda_{45}}, 
(2.4,3.9)*={\lambda_{35}},(0.5,3.9)*={-\lambda_{25}},(-0.5,3.9)*={-\lambda_{15}},
(3.5,4.9)*={-\lambda_{56}}, (2.5,4.9)*={-\lambda_{46}}, (1.5,4.9)*={-\lambda_{36}},
\end{xy}\end{array}$$
$$\begin{array}{ccc}
(\gamma_{25},\gamma_{15})&\phantom{xx}&(\gamma_{15},\gamma_{66})\\ \\
\begin{xy}<24pt,0pt>:
(0,5)*={\bublu}="F16", (1,5)*={\bugre}="F26", (2,5)*={\bugre}="F36", (3,5)*={\bugre}="F46", 
                                                                     (4,5)*={\bugre}="F56", (5,5)*={\bured}="F66", 
(0,4)*={\bured}="F15", (1,4)*={\bured}="F25", (2,4)*={\bugre}="F35", (3,4)*={\bugre}="F45", (4,4)*={\bugre}="F55", 
(0,3)*={\bured}="F14", (1,3)*={\bublu}="F24", (2,3)*={\bubla}="F34", (3,3)*={\bubla}="F44",  
(0,2)*={\bugre}="F13", (1,2)*={\bubla}="F23", (2,2)*={\bubla}="F33", 
(0,1)*={\bugre}="F12", (1,1)*={\bubla}="F22",
(0,0)*={\bugre}="F11",
"F11";(0.5,0) **@{-}, (0.7,0)*={e_2},
"F25";(1,0.6) **@{-}, (1.2,0.6)*={e_1}, "F35";(2,1.6) **@{-}, (2.2,1.6)*={e_2}, "F45";(3,2.6) **@{-}, (3.2,2.6)*={e_3},
"F55";(4,3.6) **@{-}, (4.2,3.6)*={e_4},
"F15";(0,5.5) **@{-}, "F36";(2,5.5) **@{-}, "F46";(3,5.5) **@{-}, "F56";(4,5.5) **@{-},
"F26";(1,5.5) **@{-}, "F66";(-0.5,5) **@{-}, "F56";(4.5,4.5) **@{-},
"F14";(0,-0.5) **@{-},"F12";(1.5,1) **@{-}, "F34";(3.4,3) **@{-}, "F23";(2.5,2) **@{-}, 
"F22";"F23" **@{-},   "F13";"F23" **@{-},
"F24";"F15" **@{-},  "F24";"F34" **@{-}, "F24";"F14" **@{-},
"F55";"F46" **@{-}, "F44";"F45" **@{-}, "F25";(4.5,4) **@{-}, "F45";"F36" **@{-}, 
"F35";"F26" **@{-},"F34";"F35" **@{-}, "F25";"F16" **@{-},
"F22";"F13" **@{.}, 
"F11";"F55" **@{.}, "F13";"F22" **@{.},
"F33";"F24" **@{.}, "F23";"F14" **@{.},
"F14";"F15" **@{.}, "F12";"F45" **@{.}, "F13";"F35" **@{.},
"F55";"F56" **@{.}, "F45";"F46" **@{.}, "F44";"F35" **@{.},
"F35";"F36" **@{.}, "F25";"F34" **@{.}, "F15";"F25" **@{.}, "F25";"F26" **@{.},
(-0.2,-0.3)*={\lambda_{11}}, (1.6,1)*={e_3}, (2.6,2)*={e_4}, (3.5,3)*={e_5}, 
(4.4,4)*={\lambda_{55}}, (4.7,4.5)*={e_5}, (5.3,4.8)*={-\lambda_{66}},
(-0.2,0.6)*={\lambda_{12}}, (-0.2,1.6)*={\lambda_{13}}, (-0.4,2.9)*={\lambda_{14}},  (-0.5,5)*={e_6},
(0,5.5)*={e_1}, (1,5.5)*={e_2}, (2,5.5)*={e_3}, (3,5.5)*={e_4}, (4,5.5)*={e_5},
(1.2,1.5)*={e_1},  (2.2,2.5)*={e_2}, 
(3.5,3.9)*={\lambda_{45}}, (2.4,3.9)*={\lambda_{35}},(1.5,3.9)*={\lambda_{25}},(-0.5,3.9)*={-\lambda_{15}},
(3.5,4.9)*={-\lambda_{56}}, (2.5,4.9)*={-\lambda_{46}}, (1.5,4.9)*={-\lambda_{36}}, (0.5,4.9)*={-\lambda_{26}},
\end{xy}&&
\begin{xy}<24pt,0pt>:
(0,5)*={\bugre}="F16", (1,5)*={\bugre}="F26", (2,5)*={\bugre}="F36", (3,5)*={\bugre}="F46", 
                                                                     (4,5)*={\bugre}="F56", (5,5)*={\bured}="F66", 
(0,4)*={\bured}="F15", (1,4)*={\bugre}="F25", (2,4)*={\bugre}="F35", (3,4)*={\bugre}="F45", (4,4)*={\bugre}="F55", 
(0,3)*={\bugre}="F14", (1,3)*={\bubla}="F24", (2,3)*={\bubla}="F34", (3,3)*={\bubla}="F44",  
(0,2)*={\bugre}="F13", (1,2)*={\bubla}="F23", (2,2)*={\bubla}="F33", 
(0,1)*={\bugre}="F12", (1,1)*={\bubla}="F22",
(0,0)*={\bugre}="F11",
"F11";(0.5,0) **@{-}, (0.7,0)*={e_2},
"F25";(1,0.6) **@{-}, (1.2,0.6)*={e_1}, "F35";(2,1.6) **@{-}, (2.2,1.6)*={e_2}, "F45";(3,2.6) **@{-}, (3.2,2.6)*={e_3},
"F55";(4,3.6) **@{-}, (4.2,3.6)*={e_4},
"F16";(0,5.5) **@{-}, "F36";(2,5.5) **@{-}, "F46";(3,5.5) **@{-}, "F56";(4,5.5) **@{-},
"F26";(1,5.5) **@{-}, "F66";(-0.5,5) **@{-}, "F56";(4.5,4.5) **@{-},
"F14";(0,-0.5) **@{-},"F12";(1.5,1) **@{-}, "F34";(3.4,3) **@{-}, "F23";(2.5,2) **@{-}, 
"F22";"F23" **@{-},   "F13";"F23" **@{-},
  "F24";"F34" **@{-}, "F24";"F14" **@{-},
"F55";"F46" **@{-}, "F44";"F45" **@{-}, "F25";(4.5,4) **@{-}, "F45";"F36" **@{-}, 
"F35";"F26" **@{-},"F34";"F35" **@{-}, "F25";"F16" **@{-}, "F15";"F25" **@{-}, "F14";"F15" **@{-}, 
"F22";"F13" **@{.}, 
"F11";"F55" **@{.}, "F13";"F22" **@{.},
"F33";"F24" **@{.}, "F23";"F14" **@{.},
"F12";"F45" **@{.}, "F13";"F35" **@{.},
"F55";"F56" **@{.}, "F45";"F46" **@{.}, "F44";"F35" **@{.},
"F35";"F36" **@{.}, "F25";"F34" **@{.},  "F25";"F26" **@{.},"F24";"F15" **@{.}, "F15";"F16" **@{.},
(-0.2,-0.3)*={\lambda_{11}}, (1.6,1)*={e_3}, (2.6,2)*={e_4}, (3.5,3)*={e_5}, 
(4.4,4)*={\lambda_{55}}, (4.7,4.5)*={e_5}, (5.3,4.8)*={-\lambda_{66}},
(-0.2,0.6)*={\lambda_{12}}, (-0.2,1.6)*={\lambda_{13}}, (-0.2,2.5)*={\lambda_{14}},  (-0.5,4.9)*={-\lambda_{16}},
(0,5.5)*={e_1}, (1,5.5)*={e_2}, (2,5.5)*={e_3}, (3,5.5)*={e_4}, (4,5.5)*={e_5},
(1.2,1.5)*={e_1},  (2.2,2.5)*={e_2}, 
(3.5,3.9)*={\lambda_{45}}, (2.4,3.9)*={\lambda_{35}},(1.5,3.9)*={\lambda_{25}},(-0.3,3.9)*={\lambda_{15}},
(3.5,4.9)*={-\lambda_{56}}, (2.5,4.9)*={-\lambda_{46}}, (1.5,4.9)*={-\lambda_{36}}, (0.5,4.9)*={-\lambda_{26}},
\end{xy}
\end{array}$$


$$\begin{array}{ccc}
(\gamma_{66},\gamma_{16})&&(\gamma_{16},+\infty)\\ \\
\makebox(145,110){6 more flops in the upper row}
&&
\begin{xy}<24pt,0pt>:
(0,5)*={\bured}="F16", (1,5)*={\bugre}="F26", (2,5)*={\bugre}="F36", (3,5)*={\bugre}="F46", 
                                                                     (4,5)*={\bugre}="F56", (5,5)*={\bugre}="F66", 
(0,4)*={\bugre}="F15", (1,4)*={\bubla}="F25", (2,4)*={\bubla}="F35", (3,4)*={\bubla}="F45", (4,4)*={\bubla}="F55", 
(0,3)*={\bugre}="F14", (1,3)*={\bubla}="F24", (2,3)*={\bubla}="F34", (3,3)*={\bubla}="F44",  
(0,2)*={\bugre}="F13", (1,2)*={\bubla}="F23", (2,2)*={\bubla}="F33", 
(0,1)*={\bugre}="F12", (1,1)*={\bubla}="F22",
(0,0)*={\bugre}="F11",
"F11";(0.5,0) **@{-}, (0.7,0)*={e_2},
"F25";(1,0.6) **@{-}, (1.2,0.6)*={e_1}, "F35";(2,1.6) **@{-}, (2.2,1.6)*={e_2}, "F45";(3,2.6) **@{-}, (3.2,2.6)*={e_3},
"F55";(4,3.6) **@{-}, (4.2,3.6)*={e_4}, (5.2,4.6)*={e_5}, "F66"; (5,4.5) **@{-}, "F66"; (5.5,5) **@{-},
"F16";(0,5) **@{-}, "F36";(2,5) **@{-}, "F46";(3,5) **@{-}, "F56";(4,5) **@{-},
"F26";(1,5) **@{-}, "F66";(0,5) **@{-}, 
"F14";(0,-0.5) **@{-},"F12";(1.5,1) **@{-}, "F34";(3.4,3) **@{-}, "F23";(2.5,2) **@{-}, 
"F11";"F16" **@{-},
"F22";"F26" **@{-},   
"F33";"F36" **@{-},
"F44";"F46" **@{-},
"F55";"F56" **@{-},
"F13";"F33" **@{-},
 "F14";"F44" **@{-}, 
 "F15";(4.5,4) **@{-}, 
"F11";"F66" **@{.},
"F12";"F56" **@{.}, 
"F13";"F46" **@{.},
"F13";"F46" **@{.},
"F14";"F36" **@{.},
"F15";"F26" **@{.},
"F22";"F13" **@{.}, 
"F23";"F14" **@{.},
"F33";"F15" **@{.}, 
"F34";"F16" **@{.},
 "F44";"F26" **@{.},
 "F55";"F46" **@{.},
 "F45";"F36" **@{.},
(-0.2,-0.3)*={\lambda_{11}}, (1.6,1)*={e_3}, (2.6,2)*={e_4}, (3.5,3)*={e_5}, 
(4.6,4)*={e_{6}}, 
(-0.2,0.6)*={\lambda_{12}}, (-0.2,1.6)*={\lambda_{13}},
(-0.2,2.5)*={\lambda_{14}}, (-0.5,4.9)*={\lambda_{16}},
(1.4,5.3)*={\lambda_{26}}, (2.4,5.3)*={\lambda_{36}}, (3.4,5.3)*={\lambda_{46}}, 
(4.4,5.3)*={\lambda_{56}}, (5.4,5.3)*={\lambda_{66}},
(-0.2,3.6)*={\lambda_{15}},
\end{xy}
\end{array}$$

Let us note that similar diagrams are provided in
\cite{FuMukaiFlops}. The method used in that paper is similar to ours
since the starting point is Hilb-Chow resolution but there each step
involves several flops. However, the resulting diagrams in
\cite{FuMukaiFlops} are not quite correct since they imply that the
components of the exceptional fiber in the final chamber are one
$\bP^2$ and all the rest $F_1$'s.

\section{Appendix}

\subsection{Contraction to the nilpotent cone}\label{nilpotent-cone}
In this section we recall known facts about flag varieties of
simple Lie groups and about contractions to the nilpotent cone. This subject
is classical and well documented, see e.g.~\cite{Slodowy} or
\cite{nilorbits} with the references therein. For a more recent survey 
see also \cite{Namikawarev}. 
Our point of view is
somehow more geometric, related to homogeneous varieties, in the
spirit of \cite{Ottaviani}, and directed on understanding the picture
at the level of the related root systems. We refer to
\cite[Ch.~18]{TauvelRupert} for generalities on root systems.

Let $G$ be a complex simple algebraic group with the Lie algebra
$\g$. By $R$ we denote the set of roots of $\g$ and consider the
lattices of roots and of weights $\Lambda_R \subset \Lambda_W$ of the
algebra (or group) in question; let $V=\Lambda_R\otimes\R$. By
$B$ we denote a Borel subgroup of $G$ and $F=G/B$ is its flag variety.
It is known that we have a natural isomorphism $\Pic F\iso\Lambda_W$
under which $\Nef(F)\subset N^1(F)$ is identified with the Weyl
chamber in $V$. Under this identification any irreducible
representation $U_w$ of $G$ with the highest weight $w$ corresponds to the
complete linear system on $F$ of a nef line bundle whose associated
map, $F\ra \bP(U_w)$, maps $F$ to the unique closed orbit.

Moreover, the sum of the positive roots $\rho=\sum_{\alpha\in
  R^+}\alpha$ can be identified with the anticanonical class $-K_F$
and the Weyl formula, describing the dimension of irreducible
representations, yields the Hilbert polynomial on $\Pic F$. That is, for
every $\lambda\in\Lambda_W$, the dimension formula, or the Euler
characteristic of the respective line bundle on $F$, can be written as
a polynomial
$$H(\lambda)=\prod_{\alpha\in R^+} 
\frac{((\lambda+\rho/2),\alpha)}{(\rho/2,\alpha)}$$ where $(\ \ ,\ )$
denotes the Killing form and $R^+$ is the set of positive roots.  Note
that the above polynomial is of degree equal to $\dim F$;
$H(-\lambda-\rho)=(-1)^{\dim F}H(\lambda)$ is exactly Serre duality.

The Killing form allows to relate $V$ to its dual. For every root
$\alpha\in R$ we set $V^*\ni\alpha^\vee=(v\mapsto
2(\alpha,v)/(\alpha,\alpha))$. The facets of the Weyl chamber are
supported by the simple roots, that is they are hypersurfaces defined
by forms $\alpha^\vee$. 

\begin{lemma}\label{contrF}
  The extremal contraction $\hat\pi_\alpha: F\ra F_\alpha$ associated with
  the facet $\alpha^\perp\cap\Nef(F)$ is a $\bP^1$ bundle and
  $\alpha^\vee$ is the class of the extremal curve in $N_1(F)$. 
  In $\Pic(F) = \Lambda_W$ the
  class of the relative cotangent bundle, $\Omega(F/F_\alpha)$,  is $-\alpha$.
\end{lemma}

\begin{proof}
  Note that the restriction of the polynomial $H(\lambda)$ to the
  hyperplane $\alpha^\perp$ defined by the face $\alpha^\vee$ is of degree
  $\dim F-1$ and $\alpha^\vee(\rho)=2$,
  \cite[18.7.6]{TauvelRupert}. This means that the extremal
  contraction $F\ra F_\alpha$ associated with the facet
  $\alpha^\perp\cap\Nef(F)$ is a $\bP^1$ bundle and $\alpha^\vee$ is
  the class of the fiber. On the other hand,
  $\rho-\alpha\in\alpha^\perp$ and
  $H(s_\alpha(\lambda))-\alpha)=-H(\lambda)$, which is the relative
  duality
\end{proof}

Let $X$ be the total space of the cotangent bundle of $F$, that is
$X=\Spec_F(\Symm(TF))$. Recall that $TF=G \times_B \g/\fb$, where
$\fb\subset \g$ is tangent to $B$ and $B$ acts on $\g/\fb$ via adjoint
representation and the quotient $\g\ra\g/\fb$. Alternatively,
$T^*F=G\times_B \fu$ where $\fu\subset \g$ is the nilradical of
$\fb$. The variety $X$ is symplectic. 

Since $TF$ is spanned by its
global sections,which form the Lie algebra $\g$, we have a map $X\ra\g^*$ 
which contracts the zero section to $0$. The image is a normal variety called the
nilpotent cone,  which we denote it by $Y$.  The map 
$\pi: X\ra Y$ is a symplectic contraction.

\smallskip
Clearly, $N^1(X/Y)=N^1(F)$, $\Nef(X/Y)=\Nef(F)$ and every extremal
contraction $\hat\pi_\alpha: F\ra F_\alpha$, which is a $\bP^1$
bundle, extends to a divisorial contraction $\pi_\alpha: X\ra
X_\alpha$ with all nontrivial fibers being $\bP^1$. Let
$E_\alpha\subset X$ be the exceptional divisor of $\pi_\alpha$ and
$C_\alpha$ be a general fiber of $\pi_\alpha$ restricted to
$E_\alpha$.

\begin{lemma}\label{class-alpha}
  The class of $C_\alpha$ in $V^*=N_1(X/Y)$ is $\alpha^\vee$. The
  class of $E_\alpha$ in $\Pic X = \Lambda_W$ is $-\alpha$.
\end{lemma}
\begin{proof}
We have an exact sequence of vector bundles over $F$:
$$0\lra \hat\pi_\alpha^*(\Omega F_\alpha)\lra\Omega F\lra 
\Omega(F/F_\alpha)\lra 0$$ and the divisor $E_\alpha$ in the total
space of $\Omega F$ is the total space of the sub-bundle
$\pi_\alpha^*(\Omega F_\alpha)$. Thus, the restriction of its normal
to $F$ is the line bundle $\Omega(F/F_\alpha)$ hence the lemma follows
by \ref{contrF}.
\end{proof}

\begin{corollary}{\rm c.f.~\cite[(5.2)]{Hinich}} In the above
  situation, the intersection matrix $E_\alpha\cdot C_\beta$ is the
  negative of the Cartan matrix of the respective root system.
\end{corollary}

The above observation is the key for Brieskorn-Slodowy result on the
type  singularity of the nilpotent cone in of codimension $2$; it can be
expressed as follows:

\begin{theorem}\label{BrieskornSlodowy}{\rm (Brieskorn, Slodowy)}
  Let $\pi: X=G/B \ra Y$ be the contraction to the nilpotent cone.  If
  the root system of $G$ is of type $\bA_n$, $\bD_n$, $\bE_6$,
  $\bE_7$, $\bE_8$ then in codimension 2 the contraction $\pi$ is the
  resolution of a surface Du Val singularity of the same $\bA-\bD-\bE$
  type. If $G$ is of type $\bB_n$, $\bC_n$, $\bF_4$ and $\bG_2$ then
  in codimension 2 the contraction $\pi$ is the resolution of
  singularities of type $\bA_{2n-1}$, $\bD_{n+1}$, $\bE_6$ and $\bD_4$
  and the irreducible components of the exceptional set of $\pi$ are
  in bijection with the orbits of the action of the group of
  automorphisms of the Dynkin diagrams of latter type.
\end{theorem}

We have the following immediate consequence of \ref{contrF} and
\ref{class-alpha}.

\begin{corollary}\label{MovNefWeyl}
  In the above case $\Mov(X/Y)=\Nef(X/Y)$ coincides with the Weyl
  chamber.
\end{corollary}

\subsection{ Resolving $\bC^4/\sigma_3$ } \label{sigma3}

In this last section we will give an explicit  description of the symplectic resolution of the
quotient $\C^4 / \sigma_3$ introduced in section \ref{sigma3a} (see also \cite{AndreattaWisniewski2}). We 
will constantly refer to the following commutative
diagram, which comes from the presentation of $\sigma_3=D_6$ in terms
of a semisimple product, $\sigma_3 = \bZ_3\rtimes\bZ_2$:
\begin{eq}\label{resolve-sigma3}
\xymatrix{ & W  \ar[d]_{p_1}  \ar[r]^{\nu} & Z\ar[d]_{p_2}\ar[dr] \\
 &T  \ar[d]_{q}  \ar[r]&T/\Z_2  \ar[d]&X\ar[dl]^{\pi}\\
\C^4 \ar[r] & \C^4/ \Z_3 \ar[r] & \C^4 / \sigma_3}
\end{eq}

The vertical map $q: T \ra \C^4/\Z_3$ is the toric resolution of $ \C^4/ \Z_3$
which can be described as follows. Let $N_0$ be a lattice with the
basis $e_1,e_2,f_1,f_2$ and in $N_0\otimes\bR$ take the standard cone
$\langle e_1,e_2,f_1,f_2 \rangle$ representing $\C^4$.  The toric
singularity $\C^4/\Z_{3}$ is obtained by extending $N_0$ to an
overlattice $N$ (keeping the same cone) generated by adding to $N_0$
an extra generator $v_1=(e_1+e_2)/3+2(f_1+f_2)/3$. If
$v_2=2(e_1+e_2)/3+(f_1+f_2)/3$ then the rays generated by $e_i$'s,
$f_i$'s and $v_i$'s are in the fan of the toric resolution of
$\C^4/\Z_3$ which is presented in the following picture by taking a
affine hyperplane section of the cone $\langle e_1,e_2,f_1,f_2
\rangle$. The solid edges are the boundary of the cone while its
division is marked by dotted line segments.
$$
\begin{xy}<16pt,0pt>:
(0,0)*={\bullet}="0", (0,-0.5)*={2e_1},
(6,-1)*={\bullet}="1",(6.5,-1.2)*={2f_1},
(9,3)*={\bullet}="2", (9.5,3.2)*={2f_2},
(2,5)*={\bullet}="3", (1.5,5.1)*={2e_2},
(3.16,2)*={\bullet}="4", (3,1.5)*={v_2},
(5.33,1.5)*={\bullet}="5", (5.1,1)*={v_1},
"0" ; "1" **@{-},
"0" ; "2" **@{--},
"0" ; "3" **@{-}, 
"1" ; "2" **@{-}, 
"1" ; "3" **@{-}, 
"2" ; "3" **@{-}, 
"4" ; "0" **@{.}, 
"4" ; "1" **@{.}, 
"4" ; "2" **@{.}, 
"4" ; "3" **@{.}, 
"5" ; "0" **@{.}, 
"5" ; "1" **@{.}, 
"5" ; "2" **@{.}, 
"5" ; "3" **@{.}, 
"4" ; "5" **@{.}
\end{xy}
$$
The exceptional set of this resolution consists of two
divisors, $E_1, E_2$, both isomorphic to a $\proj^2$- bundle over
$\pu$, namely $\proj(\cO(2)\oplus\cO\oplus\cO)$.  They intersect along
a smooth quadric $\pu \times \pu$. 

The action of $\Z_2$
on $\bC^4/\bZ_3$ can be lifted up to an action on $T$. This action,
which is induced by the reflections in $\sigma_3=D_6$, identifies the
two divisors by identifying the $\proj^2$ ruling of $E_1$ with that of
$E_2$; on the intersection the action interchanges the coordinates on
$\pu \times \pu$.

Going back to the diagram \ref{resolve-sigma3},  $p_2$ is the resolution of the quotient $T/\Z_2$ obtained by
blowing up the surface which is the locus of $A_1$-singularities. The
morphism $p_1$ is the blow-up along the fixed point set of the
$\Z_2$-action. We denote by $\Delta _ W$ and $\Delta_Z$ the
exceptional divisors.  Then $\nu$ is a $2:1$ cover ramified along
$\Delta_W$.

The divisor $\Delta_W$ is irreducible and its intersection with the
fiber over the special point, which is the strict transformof $E_1\cup
E_2$, call it  $E_1'\cup
E_2'$, is equal to the 3rd Hirzebruch surface $F_3$. This follows from
computing the normal of the curve which is the fixed point set of the
$\Z_2$ action in the exceptional locus of $T$. Indeed, the normal of
the intersection $E_1\cap E_2=\pu \times \pu$ is $\cO(1,-2)+\cO(-2,1)$
and the normal of the diagonal in the intersection is $\cO(2)$. Thus
the normal of the diagonal of $\pu \times \pu$ in $T$ is
$\cO(-1)\oplus\cO(-1)\oplus\cO(2)$ and, since its normal in the fixed
point set is $\cO(-1)$, it follows that the normal of the fixed point
set over the diagonal is $\cO(-1)\oplus\cO(2)$. Finally, let us note
that the intersection in $W$ of the $F_3$ surface with the strict
transform of $\pu\times\pu$ is the exceptional curve 
in the surface $F_3$ and the diagonal in $\pu\times\pu$.

The fibers of the ruling $E_i \ra \pu$, for $i=1,\ 2$,  are blown up in
$E_i'$ to ruled surfaces (1st Hirzebruch) and the map $E_i'\ra \pu$
can be factored either by blow down $E_i'\ra E_i$ or by a $\pu$-bundle
$E_i'\ra F_3$.

The strict transform of the surface $E_1\cap E_2=\pu\times\pu$ is
mapped via the quotient map $W\ra Z$ to $\proj^2$, and this is a
double covering ramified over the diagonal in $\pu\times\pu$.  The
exceptional curve of $F_3$, which is the diagonal in $\pu\times\pu$,
becomes a conic in $\proj^2$.  Thus, eventually, we see that $E_1'$ is
identified with $E_2'$ and via $\nu$ they are sent to a (non-normal) divisor $E_Z$ in $Z$. The
divisors $\Delta_W$ and $E_Z$ generate $Pic Z$ and $K_Z=E_Z$.

From the computation of the intersection of curves and divisors we see
that the divisor $E_Z$ is not numerically effective, hence $Z$ admits
a birational Fano-Mori contraction $Z\ra X$ with exceptional divisor
$E_Z$. We describe the contraction by looking at the normalization of
$E_Z$. Namely, by looking at the numerical classes of curves, we
conclude that the resulting map is a composition $E_1'\ra F_3\ra S_3$,
where the latter map is the contraction of the exceptional curve in $F_3$
to the vertex of the cubic cone $S_3$.  Therefore a general fiber of
$Z\ra X$ over $E_Z$ is a $\pu$ --- that is, generally this is a
blow-down of the divisor $E_Z$ to a surface --- while the special
fiber is a $\proj^2$.  Such a contraction was discussed in
\cite{AndreattaWisniewski} where it was proved that the image $X$ is a
smooth $4$-fold and the divisor $E_Z \subset Z$ is blow-down to the
rational cubic cone $S_3 \subset X$. Moreover $K_X = \cO_X$.

Let us finally consider the induced map $\pi : X \ra Y:= \C^4 /
\sigma_3$. It is a crepant contraction which contracts the divisor
$\Delta_W$ to a surface $S$ which, outside the point $0$, is a smooth
surface of $A_1$ singularities (coming from the $\Z_2$-action);
moreover it contracts $S_3$ to $0$. The surface $S$ is non-normal in
$0$. This is a crepant, hence symplectic, resolution of $\C^4 /
\sigma_3$.

\smallskip Note that $Pic(X/Y) = \Z$, therefore $\Mov(X/Y)$ is one
dimensional. This is the only SQM model over $Y$.

\smallskip We conclude with the description of the family of rational
curves (of $X$ over $Y$). Let $C$ be the essential curve of the symplectic
resolution $\pi : X \ra Y:= \C^4 / \sigma_3$ and let $\cV \subset
RatCurves^n(X/Y)$ be a family containing $C$. Then $\cV$ is a smooth
surface which contains a $(-1)$-curve, which parametrizes the lines in
the ruling of $S_3$. The normalization of $S$ is a smooth surface and
$\cV$ is obtained by blowing up the point of the normalization which
stays over $0$.


\end{document}